\documentclass[11pt, a4paper, oneside]{amsart}

\usepackage[english]{babel}
\usepackage{amsmath, amsthm, amsfonts, mathrsfs, amssymb}
\usepackage{mathtools}
\mathtoolsset{centercolon}
\usepackage{booktabs}
\usepackage[shortlabels]{enumitem}
\setlist[itemize]{leftmargin=20pt}
\usepackage[colorlinks, citecolor = blue]{hyperref}
\usepackage[hmargin=3cm,vmargin=3cm]{geometry}
\usepackage[color=green!40]{todonotes}

\usepackage{color}
\usepackage{graphicx}
\usepackage{tikz-cd}
\usetikzlibrary{arrows}

\newcommand{\N}{\ensuremath{\mathbf{N}}}
\newcommand{\Z}{\ensuremath{\mathbf{Z}}}
\newcommand{\Q}{\ensuremath{\mathbf{Q}}}
\newcommand{\R}{\ensuremath{\mathbf{R}}}
\newcommand{\C}{\ensuremath{\mathbf{C}}}

\renewcommand{\P}{\ensuremath{\mathbb{P}}}

\newcommand{\mc}{\mathcal}
\newcommand{\ms}{\mathscr}

\DeclarePairedDelimiter\abs{\lvert}{\rvert}

\DeclarePairedDelimiter\cbrace\{\}
\DeclarePairedDelimiter\ha()
\DeclarePairedDelimiter{\ip}\langle\rangle
\DeclarePairedDelimiter{\nrm}\lVert\rVert

\newcommand{\nrmb}[1]{\bigl\|#1\bigr\|}
\newcommand{\absb}[1]{\bigl|#1\bigr|}
\newcommand{\hab}[1]{\bigl(#1\bigr)}
\newcommand{\cbraceb}[1]{\bigl\{#1\bigr\}}
\newcommand{\ipb}[1]{\bigl\langle#1\bigr\rangle}
\newcommand{\bracb}[1]{\bigl[#1\bigr]}

\newcommand{\nrms}[1]{\Bigl\|#1\Bigr\|}
\newcommand{\abss}[1]{\Bigl|#1\Bigr|}
\newcommand{\has}[1]{\Bigl(#1\Bigr)}
\newcommand{\cbraces}[1]{\Bigl\{#1\Bigr\}}
\newcommand{\ips}[1]{\Bigl\langle#1\Bigr\rangle}

\DeclareMathOperator{\sgn}{sgn}

\DeclareMathOperator{\loc}{loc}
\DeclareMathOperator{\supp}{supp}
\DeclareMathOperator{\ind}{\mathbf{1}}
\DeclareMathOperator{\UMD}{UMD}

\DeclareMathOperator{\HL}{HL}

\DeclareMathOperator{\BHT}{BHT}
\DeclareMathOperator{\BHF}{BHF}
\DeclareMathOperator{\pv}{p.v.}
\DeclareMathOperator{\ch}{ch}
\DeclareMathOperator{\parent}{\pi}

\newcommand{\dd}{\hspace{2pt}\mathrm{d}}

\newcommand{\comp}{\mathsf{c}}

\renewcommand{\l}{\ensuremath{\ell}}
\renewcommand{\emptyset}{\varnothing}
\def\avint_#1{\mathchoice{\mathop{\kern 0.2em\vrule width 0.6em height 0.69678ex depth -0.58065ex \kern -0.8em \intop}\nolimits_{\kern -0.4em#1}}{\mathop{\kern 0.1em\vrule width 0.5em height 0.69678ex depth -0.60387ex \kern -0.6em \intop}\nolimits_{#1}} {\mathop{\kern 0.1em\vrule width 0.5em height 0.69678ex depth -0.60387ex \kern -0.6em \intop}\nolimits_{#1}} {\mathop{\kern 0.1em\vrule width 0.5em height 0.69678ex depth -0.60387ex \kern -0.6em \intop}\nolimits_{#1}}}

\newtheorem{theorem}{Theorem}
\newtheorem{corollary}[theorem]{Corollary}
\newtheorem{lemma}[theorem]{Lemma}
\newtheorem{proposition}[theorem]{Proposition}

\newtheorem{conjecture}[theorem]{Conjecture}

\theoremstyle{remark}
\newtheorem{remark}[theorem]{Remark}
\newtheorem{example}[theorem]{Example}

\theoremstyle{definition}
\newtheorem{definition}[theorem]{Definition}
\newtheorem{convention}[theorem]{Convention}

\numberwithin{theorem}{section}
\numberwithin{equation}{section}
\title{Sparse domination implies vector-valued sparse domination}
\author{Emiel Lorist and Zoe Nieraeth}
\thanks{The first author is supported by the VIDI subsidy 639.032.427 of the Netherlands Organisation for Scientific Research (NWO)}

\address{Delft Institute of Applied Mathematics \\ Delft University of Technology \\ P.O. Box 5031\\ 2600 GA Delft \\The Netherlands}
\email{e.lorist@tudelft.nl}

\address{Fakult\"at f\"ur Mathematik \\ Institut f\"ur Analysis \\ Karlsruher Institut f\"ur Technologie \\ Englerstrasse 2 \\ 76131 Karlsruhe \\
Germany}
\email{znieraeth@bcamath.org}
\allowdisplaybreaks

\begin{document}
\begin{abstract}
We prove that scalar-valued sparse domination of a multilinear operator implies vector-valued sparse domination for tuples of quasi-Banach function spaces, for which we introduce a multilinear analogue of the $\UMD$ condition. This condition is characterized by the boundedness of the multisublinear Hardy-Littlewood maximal operator and goes beyond examples in which a $\UMD$ condition is assumed on each individual space and includes e.g. iterated Lebesgue, Lorentz, and Orlicz spaces.
Our method allows us to obtain sharp vector-valued weighted bounds directly from scalar-valued sparse domination, without the use of a Rubio de Francia type extrapolation result.

We apply our result to obtain new vector-valued bounds for multilinear Calder\'on-Zygmund operators as well as recover the old ones with a new sharp weighted bound. Moreover, in the Banach function space setting we improve upon recent vector-valued bounds for the bilinear Hilbert transform.
\end{abstract}




\maketitle

\section{Introduction}
Vector-valued extensions of operators prevalent in the theory of harmonic analysis have been actively studied in the past decades. A centerpoint of the theory is the result of Burkholder \cite{Bu83} and Bourgain \cite{Bo83} that the Hilbert transform on $L^p(\R)$ extends to a bounded operator on $L^p(\R;X)$ if and only if the Banach space $X$ has the so-called $\UMD$ property. From this connection one can derive the boundedness of the vector-valued extension of many operators in harmonic analysis, like Fourier multipliers and Littlewood--Paley operators.

In the specific case where $X$ is a Banach function space, i.e. a lattice of functions over some measure space, very general extension theorems are known. These follow from the deep result of Bourgain \cite{Bo84} and Rubio de Francia \cite{Ru86} on the connection between the boundedness of the lattice Hardy--Littlewood maximal operator on $L^p(\R^d;X)$ and the $\UMD$ property of $X$.
The boundedness of the lattice Hardy--Littlewood maximal operator often allows one to use the scalar-valued arguments to show the boundedness of the vector-valued extension of an operator, using very elaborate Fubini-type techniques. Moreover it connects the extension problem to the theory of Muckenhoupt weights.
 Combined this enabled Rubio de Francia to show a very general extension principle in \cite{Ru86}, yielding vector-valued extensions for operators on $L^p(\R^d)$ satisfying bounds with respect to these Muckenhoupt weights to any $\UMD$ Banach function space.
This result was  subsequently extended by Amenta, Veraar and the first author in \cite{ALV17} to a rescaled setting and by both authors in \cite{LN19} to a multilinear setting. In this latter result, sufficient conditions were given to extend a bounded multilinear operator
$$T\colon L^{p_1}(\R^d,w_1)\times\cdots\times L^{p_m}(\R^d,w_m) \to L^{p}(\R^d,w)$$
to a bounded multilinear operator
$$\widetilde{T}\colon L^{p_1}(\R^d,w_1;X_1)\times\cdots\times L^{p_m}(\R^d,w_m;X_m) \to L^{p}(\R^d,w;X),$$
where each of the (quasi-)Banach function spaces $X_j$ satisfies some rescaled $\UMD$ condition and each weight $w_j$ some Muckenhoupt condition.

From the point of view of weights, it was made clear by Li, Martell, and Ombrosi in \cite{LMO18} that rather than assuming a condition on each individual weight, it is more appropriate to consider the multilinear weight classes characterized by the multisublinear analogue of the Hardy Littlewood maximal operator, introduced by Lerner, Ombrosi, P\'erez, Torres, and Trujillo-Gonz\'alez in \cite{LOPTT09}. Subsequently, through the extrapolation theorems of the second author \cite{Ni19} and Li, Martell, Martikainen, Ombrosi, and Vuorinen \cite{LMMOV19}, it was shown that these weight classes allow one to handle vector-valued extensions with Banach function spaces outside of the class of $\UMD$ spaces, such as $\ell^\infty$. However, these methods do not exceed the example of iterated $L^q$-spaces.

\bigskip

Our main goal is to prove a multilinear extension theorem in which we use the multilinear structure to its fullest. Just as for the weights, we will impose a condition on the tuple of Banach function spaces $(X_1,\ldots, X_m)$ rather than a condition on each $X_j$ individually. In parallel to the weighted theory, we will introduce this condition using the boundedness of a certain rescaled multisublinear Hardy-Littlewood maximal operator. In the linear case this condition reads as follows:
\[
\nrmb{\widetilde{M}_{(1,1)}(f,g)}_{L^1(\R^d;L^1(\Omega))}\lesssim\|f\|_{L^p(\R^d;X)}\|g\|_{L^{p'}(\R^d;X^\ast)}
\]
for all $f \in L^p(\R^d;X)$, $g \in L^{p'}(\R^d;X^\ast)$ and some $p \in (1,\infty)$, where $\widetilde{M}_{(1,1)}$ is the bisublinear lattice maximal operator introduced in Section~\ref{section:multiHL}. In Section~\ref{section:multiUMD} we will show that this condition is equivalent to the $\UMD$ condition for Banach function spaces and motivated by this result, we will call our multilinear analog a multilinear $\UMD$ condition, although our definition only makes sense for tuples of Banach \emph{function} spaces.

Both the Banach function space extension principle from \cite{ALV17, LN19,Ru86} and the iterated $L^q$-space extension principle using the extrapolation results in \cite{LMMOV19, Ni19} use the weighted boundedness of a multilinear operator
$$T\colon L^{p_1}(\R^d,w_1)\times\cdots\times L^{p_m}(\R^d,w_m) \to L^{p}(\R^d,w)$$
to deduce the weighted boundedness of its extension
$$\widetilde{T}\colon L^{p_1}(\R^d,w_1;X_1)\times\cdots\times L^{p_m}(\R^d,w_m;X_m) \to L^{p}(\R^d,w;X).$$
Initiated by the sparse representation and domination results for Calder\'on--Zygmund operators of Hyt\"onen \cite{Hy12} and Lerner \cite{Le13a}, such weighted bounds for an operator $T$ are in recent years  often deduced from a sparse domination result for $T$. So to deduce the weighted boundedness of the vector-valued extension $\widetilde{T}$ of an operator $T$ one typically goes through implications (\hyperref[arrows]{1}) and (\hyperref[arrows]{3}) in the following diagram\phantomsection\label{arrows}
\begin{center}
  \begin{tikzpicture}
  \node (SD)[text width=4.5cm] at (0,2) {{\parbox{4.2cm}{Sparse domination for $T$}}};
  \node (VVSD)[text width=4.5cm] at (0,0) {{\parbox{4.2cm}{Sparse domination for $\widetilde{T}$}}};
  \node (WB)[text width=4.5cm]  at (6,2) {{\parbox{4cm}{Weighted bounds for $T$}}};
  \node (VVWB)[text width=4.5cm]  at (6,0) {{\parbox{4cm}{Weighted bounds for $\widetilde{T}$}}};

  \draw[-implies,double equal sign distance, shorten <=1pt, shorten >=5pt] (SD) to node [shift={(-0.1,-0.4)}]{(1) }  (WB) ;
  \draw[-implies,double equal sign distance, shorten <=1pt, shorten >=5pt] (VVSD) to node [shift={(-0.1,0.4)}]{(4)} (VVWB);
  \draw[-implies,double equal sign distance, shorten <=5pt, shorten >=5pt] (SD) to node [shift={(0.4,0)}]{(2)} (VVSD);
  \draw[-implies,double equal sign distance, shorten <=5pt, shorten >=5pt] (WB) to node [shift={(-0.4,0)}]{(3)} (VVWB);
\end{tikzpicture}
\end{center}
in which arrows (\hyperref[arrows]{1}) and (\hyperref[arrows]{4}) are well-known and unrelated to the operator $T$, see e.g \cite{LMS14, LN15, LMO18, Ni19}.
In this paper we will deduce the weighted boundedness of the vector-valued extension $\widetilde{T}$ of $T$ through implications (\hyperref[arrows]{2}) and (\hyperref[arrows]{4}) in this diagram. In particular we will show that
scalar-valued sparse domination implies vector-valued sparse domination (implication (\hyperref[arrows]{2})) with respect to tuples of spaces satisfying our multilinear $\UMD$-condition. Such a result was established by Culiuc, Di Plinio, and Ou in \cite{CDO17} for  $\ell^q$-spaces with $q\geq 1$, which in particular satisfy our multilinear $\UMD$ condition. We point out that even in the linear case $m=1$ the result of obtaining vector-valued extensions of operators in $\UMD$ Banach function spaces from sparse domination without appealing to a Rubio de Francia type extrapolation theorem is  new.

The advantage of the route through implications (\hyperref[arrows]{2}) and (\hyperref[arrows]{4}) over the route through implications (\hyperref[arrows]{1}) and (\hyperref[arrows]{3}) is that for general tuples of quasi-Banach function spaces the Fubini-type techniques needed for implication (\hyperref[arrows]{2}) are a lot less technical than the ones needed for implication (\hyperref[arrows]{3}). Moreover implication (\hyperref[arrows]{4}) yields quantitative and in many cases sharp weighted estimates for $\widetilde{T}$, while the weight dependence in the arguments used for implication (\hyperref[arrows]{3}) is not easily tracked and certainly not sharp.
 A downside of our approach through implications (\hyperref[arrows]{2}) and (\hyperref[arrows]{4}) is the fact that we need sparse domination for $T$ as a starting point, while one only needs weighted bounds in order to apply (\hyperref[arrows]{3}). It therefore remains an interesting open problem whether (\hyperref[arrows]{3}) also holds under our multilinear $\UMD$-condition, rather than a $\UMD$ condition on each individual space as in \cite{LN19}.

\bigskip

 Our main result relies on the following two key ingredients:
 \begin{itemize}
   \item The equivalence between sparse form and the $L^1$-norm of the multisublinear maximal function. This equivalence seems to have been used for the first time in \cite{CDO17} by Culiuc, Di Plinio, and Ou.
   \item  A sparse domination result for the multisublinear lattice maximal operator under the multilinear $\UMD$ condition assumption. For this we extend the idea of H\"anninen and the first author in \cite{HL17}, where a linear version of this result was obtained.
 \end{itemize}
Combining these ingredients, we obtain the following theorem in the linear case:
\begin{theorem}\label{theorem:mainlinear}
  Let $T$ be an operator such that for any $f,g \in L^\infty_c(\R^d)$ there exists a sparse collection of cubes $\mc{S}$ such that
  \begin{equation*}
    \int_{\R^d}\!|Tf|\cdot |g| \dd x \leq C_T\, \sum_{Q\in \mc{S}}\ip{f}_{1,Q}  \ip{g}_{1,Q} \abs{Q}
  \end{equation*}
  Let $X$ be a $\UMD$ Banach function space over a  measure space $(\Omega,\mu)$ and suppose that for any simple function  $f \in L^\infty_c(\R^d;X)$ the function $\widetilde{T}f:\R^d \to X$ given by
  \begin{equation*}
    \widetilde{T}f(x,\omega):= T(f(\cdot,\omega))(x), \qquad (x,\omega) \in \R^d \times \Omega
  \end{equation*}
  is well-defined and strongly measurable. Then for all simple functions $f \in L^\infty_c(\R^d, X)$  and $g \in L^\infty_c(\R^d)$ there exists a sparse collection of cubes $\mc{S}$ such that
  \begin{equation*}
    {\int_{\R^d}\!\|\widetilde{T}f\|_X\cdot \abs{g} \dd x} \lesssim_{X} C_T \sum_{Q\in \mc{S}} \ipb{\nrm{f}_X}_{1,Q} \ipb{\abs{g}}_{1,Q}\abs{Q} .
  \end{equation*}
\end{theorem}
 Note that if $T$ is linear, then for any simple $f \in L^\infty_c(\R^d)$ we have $\widetilde{T}f = (T\otimes I_X) f$, which is always well-defined and strongly measurable.

Theorem \ref{theorem:mainlinear} can be generalized to operators that are sparsely dominated by forms
\begin{equation*}
  \sum_{Q\in \mc{S}} \ip{f}_{r,Q} \ip{g}_{s',Q} \abs{Q},
\end{equation*}
for $0<r<s\leq \infty$, in which case we have to replace the $\UMD$ condition on $X$ by the rescaled $\UMD_{r,s}$ condition, see Section~\ref{section:multiUMD}. A rescaled $\UMD$ condition was already used in the previous work \cite{LN19} where the condition that $((X^r)^*)^{(s'/r)'}$ has the $\UMD$ property was imposed. However, the condition $X\in\UMD_{r,s}$ we impose in this work is seemingly weaker if $r \geq 1$ and we refer to Proposition~\ref{proposition:comparisonoldassumption} for a comparison.

For the multilinear version of this result we require the $\UMD_{\vec{r},s}$ condition imposed on an $m$-tuple of Banach spaces $\vec{X}=(X_1,\ldots,X_m)$, which will be introduced in Section \ref{section:multiUMD}. This yields the following more general version of Theorem \ref{theorem:mainlinear}, which is our main result.
\begin{theorem}\label{theorem:mainmultilinear}
Let $\vec{{r}}\in(0,\infty)^{m}$, $s \in (1,\infty]$ and let $T$ be an operator defined on $m$-tuples of functions such that for any $\vec{f},g \in L^\infty_c(\R^d)$  there exists a sparse collection $\mc{S}$ such that
\begin{equation*}
  \int_{\R^d}\! |T(\vec{f}\hspace{2pt})| \cdot |g| \dd x \leq C_T\, \sum_{Q\in \mc{S}}\Big( \prod_{j=1}^m\ip{f_j}_{r_j,Q}\Big) \ip{g}_{s',Q}|Q|.
\end{equation*}
Let $\vec{X}$ be an $m$-tuple of Banach function spaces over a  measure space $(\Omega,\mu)$ such that $\vec{X} \in \UMD_{\vec{r},s}$ and suppose that for all simple functions $\vec{f} \in L^\infty_c(\R^d, \vec{X})$ the function $\widetilde{T}(\vec{f}\hspace{2pt}):\R^d \to X$ given by
  \begin{equation*}
    \widetilde{T}(\vec{f}\hspace{2pt})(x,\omega):= T(\vec{f}(\cdot,\omega))(x), \qquad (x,\omega) \in \R^d \times \Omega
  \end{equation*}
  is well-defined and strongly measurable. Let $X:=\prod_{j=1}^m X_j$. Then for all simple functions $\vec{f} \in L^\infty_c(\R^d, \vec{X})$  and $g \in L^\infty_c(\R^d)$ there exists a sparse collection of cubes $\mc{S}$ such that
  \begin{equation*}
    \int_{\R^d}\!\|\widetilde{T}(\vec{f}\hspace{2pt})\|_X\cdot \abs{g} \dd x \lesssim_{\vec{X},\vec{r},s} C_T \sum_{Q\in \mc{S}} \Big(\prod_{j=1}^m\ip{\nrm{f_j}_{X_j}}_{r_j,Q}\Big) \ipb{\nrm{g}_{X^*}}_{s',Q}\abs{Q}.
  \end{equation*}
\end{theorem}
Note that we only allow Banach function spaces in Theorem \ref{theorem:mainmultilinear}. However, in the multilinear setting it is natural to expect estimates also for quasi-Banach function spaces. We are able to consider these spaces by the more general Theorem~\ref{thm:mainthm1}, which is facilitated by introducing a rescaling parameter. For a discussion on this we refer the reader to Section~\ref{sec:applymain}.

By the known sharp weighted bounds for the sparse forms we can deduce vector-valued weighted bounds as a corollary from our main result, which is the extension theorem we were after.
Note that these bounds are new even in the case $\vec{w} \equiv 1$.

\begin{corollary}\label{corollary:weights1}
Assume the conditions of Theorem \ref{theorem:mainmultilinear} and additionally suppose that $T$ is either $m$-linear or positive-valued $m$-sublinear. Then for all $p \in (0,\infty]^m$ with $\vec{p}>\vec{r}$ and $\frac{1}{p} := \sum_{j=1}^m\frac{1}{p_j}>\frac{1}{s}$, and all $\vec{w}\in A_{\vec{p},(\vec{r},s)}$ we have
\begin{equation*}
\nrmb{\widetilde{T}(\vec{f}\hspace{2pt})}_{L_w^p(\R^d;X)}\lesssim_{\vec{X},\vec{p},q,\vec{r},s}C_T\,[\vec{w}]^{\max\cbraces{\frac{\frac{1}{\vec{r}}}{\frac{1}{\vec{r}}-\frac{1}{\vec{p}}},\frac{1-\frac{1}{s}}{\frac{1}{p}-\frac{1}{s}}}}_{\vec{p},(\vec{r},s)}\prod_{j=1}^m\|f_j\|_{L^{p_j}_{w_j}(\R^d;X_j)}
\end{equation*}
for all $\vec{f}\in L^{\vec{p}}_{\vec{w}}(\R^n;\vec{X})$.
\end{corollary}

This paper is organized as follows:
\begin{itemize}
  \item In Section~\ref{sec:prelim} we discuss the preliminaries on product quasi-Banach function spaces, sparse forms, and multilinear weight classes.
  \item In Section~\ref{section:multiHL} we introduce a rescaled multilinear analogue of the Hardy-Littlewood property and prove sparse domination of the multisublinear Hardy-Littlewood maximal operator using this condition.
  \item In Section~\ref{section:multiUMD} we introduce a limited range multilinear $\UMD$ property for tuples of quasi-Banach function spaces.
  \item In Section~\ref{sec:proofs} we state and prove our main results.
  \item In Section~\ref{sec:applymain} we discuss how our results can be applied in the quasi-Banach range as well as prove new vector-valued bounds for various operators.
\end{itemize}

\subsection*{Notation}

We consider $\R^d$ with the Lebesgue measure $\mathrm{d}x$ and write $|E|$ for the measure of a measurable set $E\subseteq\R^d$. For $r \in (0,\infty)$, a function $f\in L^r_{\loc}(\R^d)$ and a measurable $E\subseteq \R^d$ of positive finite measure we use the notation
$$\langle f\rangle_{r,E}=\has{\frac{1}{\abs{E}}\int_E |f|^r\,\mathrm{d}x}^{\frac{1}{r}}.$$ Moreover, we denote by $\langle f\rangle_{\infty,E}$ the essential supremum of $|f|$ in $E$. We let $\ind_E$ denote the characteristic function of $E$.

Throughout the paper we write $C_{a,b,\cdots}$ to denote a constant, which only depends on
the parameters $a,b,\cdots$ and possibly on the dimension $d$ and multilinearity $m$. By $\lesssim_{a,b,\cdots}$ we mean that there is a constant $C_{a,b,\cdots}$ such that inequality holds and by $\eqsim_{a,b,\cdots}$ we mean that both $\lesssim_{a,b,\cdots}$ and $\gtrsim_{a,b,\cdots}$ hold.

\begin{convention}\label{conv:prod}
  For a vector $\vec{p}\in (0,\infty]^m$ we denote its coordinates by $p_1,\cdots,p_m$ and set $\frac{1}{p}:=\sum_{j=1}^m \frac{1}{p_j}$. We denote $\max{\vec{p}} := \max\cbrace{p_1,\cdots,p_m}$ and
 $$L^{\vec{p}}(\R^d) := L^{p_1}(\R^d)\times \cdots\times L^{p_m}(\R^d).$$
For $\vec{q}\in (0,\infty)^m$ we write $\vec{p}\leq\vec{q}$ if $p_j\leq q_j$ and $\vec{p}>\vec{q}$ if $p_j>q_j$ for  $1\leq j\leq m$. We define arithmetic operations on $\vec{p}$ and $\vec{q}$ coordinate wise, e.g. $\vec{p}/\vec{q}:=(p_1/q_1,\ldots,p_m/q_m)$, and similarly $\vec{p}\hspace{2pt}^\alpha=(p_1^\alpha,\ldots,p_m^\alpha)$ for $\alpha>0$.

For an $m$-tuple of quasi-Banach function spaces $\vec{X}=(X_1,\cdots,X_m)$ and  we will write $X:=\prod_{j=1}^m X_j$, for $\vec{p}\in (0,\infty)$ and for a vector of $m$ weights $\vec{w}=(w_1,\cdots,w_m)$ (see Section \ref{subsec:sparse}) we write $w:=\prod_{j=1}^m w_j$.
Moreover we will use the shorthand notation
$$L^{\vec{p}}_{\vec{w}}(\R^d;\vec{X}) := L^{p_1}_{w_1}(\R^d;X_1)\times \cdots \times L^{p_m}_{w_m}(\R^d;X_m)$$
and $L^{\vec{p}}_{\loc}(\R^d;\vec{X})$ is defined similarly. We say $\vec{X}$ is $\vec{r}$-convex for $\vec{r}\in (0,\infty)^m$ if $X_j$ is $r_j$-convex for $1\leq j\leq m$.

\end{convention}

\section{Preliminaries}\label{sec:prelim}
\subsection{Product quasi-Banach function spaces}\label{sec:BFS}
Let $(\Omega,\mu)$ be a $\sigma$-finite measure space. An order ideal $X \subseteq L^0(\Omega)$ equipped with a quasi-norm $\nrm{\,\cdot\,}_X$ is called a \emph{quasi-Banach function space} if it satisfies the following properties
\begin{itemize}
  \item \textit{Compatibility:} If $\xi,\eta \in X$ with $\abs{\xi}\leq \abs{\eta}$, then $\nrm{\xi}_X\leq \nrm{\eta}_X$.
  \item \textit{Weak order unit:} There is a $\xi \in X$ with $\xi > 0$ a.e.
  \item \textit{Fatou property:} If $0\leq \xi_n \uparrow \xi$ for $(\xi_n)_{n=1}^\infty$ in $X$ and $\sup_{n\in \N}\nrm{\xi_n}_X<\infty$, then $\xi \in X$ and $\nrm{\xi}_X=\sup_{n\in\N}\nrm{\xi_n}_X$.
\end{itemize}
If $\nrm{\,\cdot\,}_X$ is a norm then $X$ is called a \emph{Banach function space}.

A quasi-Banach function space $X$ is called \emph{order-continuous} if for any sequence $0\leq \xi_n\uparrow \xi \in X$ we have $\nrm{\xi_n-\xi}_X \to 0$. As an example we note that all reflexive Banach function spaces are order-continuous.
Order-continuity of $X$ ensures that the Bochner space $L^p(\R^d;X)$ for $p \in (0,\infty)$ coincides with the \emph{mixed-norm space} of all measurable functions $\R^d\times \Omega \to \C$ such that
\begin{equation*}
  \nrmb{x\mapsto\nrm{f(x,\cdot)}_X}_{L^p(\R^d)} <\infty,
\end{equation*}
which is again a quasi-Banach function space. Moreover if $X$ is an order-continuous Banach function space, then its dual $X^*$ is also a Banach function space. For an introduction to Banach function spaces we refer the reader to \cite[Section 1.b]{LT79} or \cite{BS88}.

We call a quasi-Banach function space $p$-convex for $p \in (0,\infty)$ if for any $\xi_1,\cdots,\xi_n \in X$ we have
\begin{equation*}
  \nrms{\has{\sum_{k=1}^n\abs{\xi_k}^p}^{1/p}}_X \leq \has{\sum_{k=1}^n \nrm{\xi_k}^p}^{1/p}.
\end{equation*}
Often a constant is allowed in the defining inequality for $p$-convexity, but as shown in \cite[Theorem 1.d.8]{LT79} $X$ can always be renormed equivalently such that this constant equals $1$. The \emph{$p$-concavification} of $X$ for $p \in (0,\infty)$ is defined as
\begin{equation*}
  X^p:= \cbraceb{\abs{\xi}^p\sgn{\xi}:\xi \in X} = \cbraceb{\xi \in L^0(\Omega):\abs{\xi}^{1/p} \in X}
\end{equation*}
equipped with the quasinorm $\nrm{\xi}_{X^p}:= \nrm{\abs{\xi}^{1/p}}_X^p$. Note that $\nrm{\,\cdot\,}_{X^p}$ is a norm if and only if $X$ is $p$-convex. In particular for $f \in L^p_{\loc}(\R^d;X)$ and a set $E\subseteq \R^d$ of finite measure the $p$-convexity of $X$ ensures that $\ip{\abs{f}}_{p,E}$ is well-defined as a Bochner integral.
See \cite[Section 1.d]{LT79} and \cite{Ka84} for a further introduction to $p$-convexity and related notions.

For $m$ quasi-Banach function spaces $X_1,\cdots,X_m$ over the same measure space we define the product space
\begin{equation*}
  \prod_{j=1}^m X_j := \cbraces{\xi:\abs{\xi} \leq \prod_{j=1}^m \xi_j\text{ for some } 0\leq \xi_j \in X_j,\,1\leq j\leq m},
\end{equation*}
which can easily seen to be a vector space and for $\xi \in \prod_{j=1}^m X_j$ we define
\begin{equation*}
  \nrm{\xi}_{\prod_{j=1}^m X_j} := \inf \cbraces{\prod_{j=1}^m \nrm{\xi_j}_{X_j}:\abs{\xi} \leq \prod_{j=1}^m \xi_j, 0\leq \xi_j \in X_j,\,1\leq j\leq m}.
\end{equation*}
We call $\vec{X} = (X_1,\cdots,X_m)$ an \emph{$m$-tuple of quasi-Banach function spaces} if $X_1,\cdots,X_m$ are quasi-Banach function spaces over the same measure space and the product $\prod_{j=1}^m X_j$ equipped with the norm
 $\nrm{\,\cdot\,}_{\prod_{j=1}^m X_j}$ is also quasi-Banach function space. We refer to \cite{Ca64,Lo69,Sc10} for background on product Banach function spaces.
Let us give a few examples:
\begin{proposition}\label{example:products}
  Let $(\Omega,\mu)$ be a $\sigma$-finite measure space.
\begin{enumerate}[(i)]
  \item \label{it:product0} For any quasi-Banach function space $X$ we have $X\cdot L^\infty(\Omega)=X$.
  \item \label{it:product1} \textit{Lebesgue spaces:} $L^p(\Omega) = \prod_{j=1}^m L^{p_j}(\Omega)$ for $\vec{p} \in (0,\infty]^m$.
  \item \label{it:product2} \textit{Lorentz spaces:} $L^{p,q}(\Omega) = \prod_{j=1}^m L^{p_j,q_j}(\Omega)$ for $\vec{p}\in(0,\infty)^m$, $\vec{q} \in (0,\infty]^m$ .
  \item \label{it:product3} \textit{Orlicz spaces:} $L^{\Phi}(\Omega) = \prod_{j=1}^m L^{\Phi_j}(\Omega)$ for Young functions $\Phi_j$ and $\Phi^{-1}= \prod_{j=1}^m\Phi_j^{-1}$.
\end{enumerate}
In all these cases the (quasi)-norm of the product is equivalent to the usual (quasi)-norm.
\end{proposition}
\begin{proof}
The first claim follows directly from the definitions. For \ref{it:product1}, \ref{it:product2}, and \ref{it:product3}, the inclusion $\prod_{j=1}^mX_j\subseteq X$ with $X$ respectively  equal to $L^p(\Omega)$, $L^{p,q}(\Omega)$, and $L^{\Phi}(\Omega)$ and $X_j$ respectively equal to $L^{p_j}(\Omega)$, $L^{p_j,q_j}(\Omega)$, and $L^{\Phi_j}(\Omega)$, follows from the generalized H\"older's inequality $\|\prod_{j=1}^m\xi_j\|_X\lesssim\prod_{j=1}^m\|\xi_j\|_{X_j}$ valid for these spaces, see \cite{On63, On65}.

For the converse in \ref{it:product1} and \ref{it:product2} in the case that $q=q_1=\cdots=q_m=\infty$, let $\xi\in L^p(\Omega)$ or $\xi\in L^{p,\infty}(\Omega)$ respectively. If $p=p_1=\cdots=p_m=\infty$, the result follows from \ref{it:product0}. Otherwise, we set $\xi_j:=|\xi|^{\frac{p}{p_j}}$. Then $\xi_j\in L^{p_j}(\Omega)$ or $\xi_j\in L^{p_j,\infty}(\Omega)$ respectively, $|\xi|=\prod_{j=1}^m\xi_j$, and $\prod_{j=1}^m\|\xi_j\|_{L^{p_j}(\Omega)}=\|\xi\|_{L^p(\Omega)}$ or similarly in the weak case, proving the result. The converse for \ref{it:product3} is proven analogously with $\xi_j:=\Phi_j^{-1}(\Phi(|\xi|))$.

Finally, for \ref{it:product2} in the case $q_k<\infty$ for some $1 \leq k \leq j$ we take $s >0$ such that $X_j:= L^{{p_j/s},{q_j/s}}(\Omega)$ are all Banach spaces. Then $X_k$ is reflexive, which means that we can identify the product space $\prod_{j=1}^m L^{p_j/s,q_j/s}(\Omega)$ with an iterated complex interpolation space by \cite{Ca64}. So $\prod_{j=1}^m L^{p_j/s,q_j/s}(\Omega)= L^{p/s,q/s}(\Omega)$
by \cite[Theorem~1.10.3 and 1.18.6]{Tr78}. The assertion now follows by rescaling.
\end{proof}

\subsection{Sparse forms and multilinear weight classes}\label{subsec:sparse}
In this section we briefly outline some of the results on dyadic grids and sparse collections of cubes that we will use. For proofs of these result and other relevant properties we refer the reader to \cite{LN15}. Furthermore we will introduce multilinear weight classes and state some weighted results, for which we refer the reader to \cite{Ni19}. We note these results also hold in the more general setting of spaces homogeneous type, provided one uses the right notion of dyadic cubes in this setting, see \cite{HK12}.

By a \emph{cube} $Q\subseteq\R^d$ we mean a half-open cube whose sides are parallel to the coordinate axes. We define the standard dyadic grid as
\[
\ms{D}:=\bigcup_{k\in\Z}\big\{2^{-k}\big([0,1)^d+m\big):m\in\Z^d\big\}.
\]
An important property pertaining to cubes is the fact that there exist $3^d$ translates $(\ms{D}^\alpha)_{\alpha=1}^{3^d}$ of $\ms{D}$ such that for each cube $Q\subseteq\R^d$ there exists an $\alpha$ and a cube $Q'\in\ms{D}^\alpha$ such that $Q\subseteq Q'$ and $|Q'|\leq 6^d|Q|$. This so-called three lattice lemma will allow us to reduce our arguments to only having to consider dyadic grids.

A collection of cubes $\mc{S}$ is called \emph{sparse} if there is pairwise disjoint collection $(E_Q)_{Q\in\mc{S}}$ of measurable sets satisfying $E_Q\subseteq Q$ and $\abs{E_Q} \geq \frac{1}{2}|Q|$.
Note that the constant $\frac12$ in the estimate $\abs{E_Q} \geq \frac{1}{2}|Q|$ is not essential in the arguments and could be replaced by an $\eta \in (0,1)$. What is important is that this constant stays fixed throughout the arguments.

\begin{definition}
For $\vec{r}\in(0,\infty)^m$ and $\vec{f}\in L^{\vec{r}}_{\loc}(\R^d)$ we define the $m$-sublinear Hardy-Littlewood maximal operator
\[
M_{\vec{r}}(\vec{f}\hspace{2pt})(x):=\sup_Q\prod_{j=1}^m\langle f_j\rangle_{r_j,Q}\ind_Q(x), \qquad x \in \R^d
\]
where the supremum is taken over all cubes $Q\subseteq \R^d$.
Similarly, for a collection of cubes $\mc{D}$ we define
\[
M^{\mc{D}}_{\vec{r}}(\vec{f}\hspace{2pt})(x):=\sup_{Q\in\mc{D}}\prod_{j=1}^m\langle f_j\rangle_{r_j,Q}\ind_Q(x) , \qquad x \in \R^d.
\]
\end{definition}
A \emph{weight} $w$ is a measurable function $w:\R^d\to (0,\infty)$. For a weight $w$ and $p\in(0,\infty]$ we define the weighted Lebesgue space $L^p_w(\R^d)$ as the space of those measurable functions $f$ satisfying $\|fw\|_{L^p(\R^d)}<\infty$. Note that if $p\in(0,\infty)$, then $L^p_w(\R^d)$ is the $L^p$ space over $\R^d$ with respect to the measure $w^p\dd x$. It should be noted that our definition of $L^p_w(\R^d)$ is often denoted by $L^p(w^p)$ in the literature when $p<\infty$. The advantage of our definition is that we also obtain a sensible definition when $p=\infty$, which does play a role in the theory.

\begin{definition}
Let $\vec{r}\in(0,\infty)^m$, $s\in(0,\infty]$ and let $\vec{p}\in(0,\infty]^m$ with $\vec{r}\leq\vec{p}$ and $p\leq s$. Let $\vec{w}$ be a vector of $m$ weights. We say that $\vec{w}$ is a multilinear Muckenhoupt weight and write $\vec{w}\in A_{\vec{p},(\vec{r},s)}$ if
\[
[\vec{w}]_{\vec{p},(\vec{r},s)}:=\sup_Q\has{\prod_{j=1}^m\langle w_j^{-1}\rangle_{\frac{1}{\frac{1}{r_j}-\frac{1}{p_j}},Q}}\langle w\rangle_{\frac{1}{\frac{1}{p}-\frac{1}{s}},Q}<\infty,
\]
where the supremum is taken over all cubes $Q\subseteq \R^d$.
\end{definition}
These weight classes characterize the weak boundedness of the Hardy-Littlewood maximal operator in the sense that $[\vec{w}]_{\vec{p},(\vec{r},\infty)}\eqsim_d \|M_{\vec{r}}\|_{L^{\vec{p}}_{\vec{w}}(\R^d)\to L^{p,\infty}_w(\R^d)}$. In the case where we have the strict inequalities $\vec{r}<\vec{p}$ this improves to a strong bound. More precisely, we have the following result, which is shown in \cite[Proposition 2.14]{Ni19}.
\begin{proposition}\label{prop:multimaxwest}
Let $\vec{r}\in(0,\infty)^m$ and $\vec{p}\in(0,\infty]^m$ with $\vec{r}<\vec{p}$. Then  for all $\vec{w}\in A_{\vec{p},(\vec{r},\infty)}$
\[
\|M_{\vec{r}}\|_{L^{\vec{p}}_{\vec{w}}(\R^d)\to L^p_w(\R^d)}\lesssim_{\vec{p},\vec{r}}[\vec{w}]_{\vec{p},(\vec{r},\infty)}^{\max\cbraces{\frac{\frac{1}{\vec{r}}}{\frac{1}{\vec{r}}-\frac{1}{\vec{p}}}}}.
\]
\end{proposition}

We wish to elaborate on the intimate connection between the multisublinear Hardy-Littlewood maximal operator and sparse forms. Indeed, for $\vec{r}\in(0,\infty)^m$ we have the equivalence
\begin{equation}\label{eq:equivalencespmax}
\|M_{\vec{r}}(\vec{f}\hspace{2pt})\|_{L^1(\R^d)}\eqsim \sup_{\mc{S}}\sum_{Q\in\mc{S}}\Big(\prod_{j=1}^m\langle f_j\rangle_{r_j,Q}\Big)|Q|,
\end{equation}
which is valid for all $\vec{f}\in L^{\vec{r}}_{\loc}(\R^d)$, where the supremum is taken over all sparse collections $\mc{S}$, see \cite[Remark 1.5]{CDO17}. Applying \eqref{eq:equivalencespmax} with $m=2$, $r_1=r_2=1$ allows us rewrite the sparse domination results in Theorem~\ref{theorem:mainlinear} in terms of multisublinear maximal functions, which is an essential step in its proof.

For the sparse domination in our main result (Theorem~\ref{thm:mainthm1}), our sparse domination is in fact written in terms of a multisublinear maximal function. We emphasize how one can view this in terms of sparse forms, being a particular case of \eqref{eq:equivalencespmax}, in the following proposition:
\begin{proposition}\label{prop:sparsequiv}
Let $\vec{r}\in(0,\infty)^m$, $q\in(0,\infty)$, and $s\in(q,\infty]$. Then for all $\vec{f}\in L^{\vec{r}}_{\loc}(\R^d)$ and $g\in L^{\frac{1}{\frac{1}{q}-\frac{1}{s}}}_{\loc}(\R^d)$
\[
\|M_{(\vec{r},\frac{1}{\frac{1}{q}-\frac{1}{s}})}(\vec{f},g)\|_{L^q(\R^d)}\eqsim_{q,r,s} \sup_{\mc{S}}\has{\sum_{Q\in\mc{S}}\Big(\prod_{j=1}^m\langle f_j\rangle_{r_j,Q}\Big)^q\langle g\rangle^q_{\frac{1}{\frac{1}{q}-\frac{1}{s}},Q}|Q|}^{\frac{1}{q}},
\]
where the supremum is over all sparse collections $\mc{S}$.
\end{proposition}
\begin{proof}
Note that the left-hand side can be written as
\[
\nrmb{M_{(\frac{\vec{r}}{q},\frac{\frac{1}{q}}{\frac{1}{q}-\frac{1}{s}})}(|f_1|^q,\ldots,|f_m|^q,|g|^q)}^{\frac{1}{q}}_{L^1(\R^d)}
\]
while the right-hand side can be written as
\[
\has{\sup_{\mc{S}}\sum_{Q\in\mc{S}}\Big(\prod_{j=1}^m\langle |f_j|^q\rangle_{\frac{r_j}{q},Q}\Big)\langle |g|^q\rangle_{\frac{\frac{1}{q}}{\frac{1}{q}-\frac{1}{s}},Q}|Q|}^{\frac{1}{q}}
\]
Thus, the result follows as a special case of \eqref{eq:equivalencespmax}.
\end{proof}
We point out that pointwise domination of an operator by a sparse operator of course implies that this operator also satisfies sparse domination in form. Since we will use this fact several times, we record it here.
\begin{proposition}\label{prop:ptwisetoform}
Let $\vec{r}\in(0,\infty)^m$, $q\in(0,\infty)$ and let $T$ be an operator defined on $m$-tuples of functions such that for $\vec{f} \in L^{\vec{r}}_{\loc}(\R^d)$ there exists a sparse collection $\mc{S}$ such that
\[
\absb{T(\vec{f}\hspace{2pt})(x)}\leq C_T\has{\sum_{Q\in\mc{S}}\Big(\prod_{j=1}^m\langle f_j\rangle_{r_j,Q}\Big)^q\ind_Q(x)}^{\frac{1}{q}},\qquad x \in \R^d.
\]
Then
\[
\|T(\vec{f}\hspace{2pt})\cdot g\|_{L^q(\R^d)}\lesssim_q C_T\,\|M_{(\vec{r},q)}(\vec{f},g)\|_{L^q(\R^d)}
\]
for all $g\in L^q_{\loc}(\R^d)$.
\end{proposition}
\begin{proof}
By Proposition~\ref{prop:sparsequiv} we have
\begin{align*}
\|T(\vec{f}\hspace{2pt})g\|_{L^q(\R^d)}&\leq C_T\Big\|\has{\sum_{Q\in\mc{S}}\Big(\prod_{j=1}^m\langle f_j\rangle_{r_j,Q}\Big)^q\ind_Q}^{\frac{1}{q}}g\Big\|_{L^q(\R^d)}\\
&=C_T\has{\int_{\R^d}\sum_{Q\in\mc{S}}\Big(\prod_{j=1}^m\langle f_j\rangle_{r_j,Q}\Big)^q\ind_Q|g|^q\,\mathrm{d}x}^{\frac{1}{q}}\\&=C_T\has{\sum_{Q\in\mc{S}}\Big(\prod_{j=1}^m\langle f_j\rangle_{r_j,Q}\Big)^q\langle g\rangle_{q,Q}^q|Q|}^{\frac{1}{q}}\\
&\lesssim_q C_T\|M_{(\vec{r},q)}(\vec{f},g)\|_{L^q(\R^d)},
\end{align*}
for all $g\in L^q_{\loc}(\R^d)$, as desired.
\end{proof}

Next we note that vector-valued sparse domination can be written in two equivalent ways. The first way uses duality in $X$, which is useful as it allows one to apply Fubini's theorem. The second is domination with the norm of $X$ on the inside, which is the form that is usually used in the literature.
\begin{proposition}\label{prop:vvsdequivalence}
Let $\vec{r}\in(0,\infty)^m$, $q\in(0,\infty)$, $s\in(q,\infty]$ and let $\vec{X}$ be an $m$-tuple of quasi-Banach function spaces over a measure space $(\Omega,\mu)$ such that $X$ is $q$-convex and order-continuous. Let $\widetilde{T}$ be an operator defined on an $m$-tuple $\vec{f}\in L^{\vec{r}}_{\loc}(\R^d;\vec{X})$ with $\widetilde{T}(\vec{f})\in L^0(\R^d;X)$. Then the following are equivalent:
\begin{enumerate}[(i)]
\item\label{it:vvsdeq1} For all $g\in L^\infty_c(\R^d;((X^q)^\ast)^{\frac{1}{q}})$
\[
\nrmb{\widetilde{T}(\vec{f}\hspace{2pt})\cdot g}_{L^q(\R^d;L^q(\Omega))}\leq C\,\nrmb{M_{(\vec{r},\frac{1}{\frac{1}{q}-\frac{1}{s}})}(\|\vec{f}\|_{\vec{X}}, \|g\|_{((X^q)^\ast)^{\frac{1}{q}}})}_{L^q(\R^d)}.
\]
\item\label{it:vvsdeq2} For all $g\in L^\infty_c(\R^d)$
\[
\nrmb{\|\widetilde{T}(\vec{f}\hspace{2pt})\|_X\cdot g}_{L^q(\R^d)}\leq C\,\nrmb{M_{(\vec{r},\frac{1}{\frac{1}{q}-\frac{1}{s}})}(\|\vec{f}\|_{\vec{X}}, g)}_{L^q(\R^d)}.
\]
\end{enumerate}
\end{proposition}
\begin{proof}
For \ref{it:vvsdeq2}$\Rightarrow$\ref{it:vvsdeq1}, note that
\begin{align*}
\|\widetilde{T}(\vec{f}\hspace{2pt})\cdot g\|_{L^q(\Omega)}&=\left(\int_\Omega\!|\widetilde{T}(\vec{f}\hspace{2pt})|^q|g|^q\,\mathrm{d}\mu\right)^{\frac{1}{q}} \leq\nrmb{|\widetilde{T}(\vec{f}\hspace{2pt})|^q}^{\frac{1}{q}}_{X^q}\nrmb{|g|^q}_{(X^q)^\ast}^{\frac{1}{q}}\\
&=\nrmb{\widetilde{T}(\vec{f}\hspace{2pt})}_X\nrm{g}_{((X^q)^\ast)^{\frac{1}{q}}}
\end{align*}
so that $\nrmb{\widetilde{T}(\vec{f}\hspace{2pt})\cdot g}_{L^q(\R^d;L^q(\Omega))}\leq\nrmb{\|\widetilde{T}(\vec{f}\hspace{2pt})\|_X\cdot \|g\|_{((X^q)^\ast)^{\frac{1}{q}}}}_{L^q(\R^d)}$. Since for $g\in L^\infty_c(\R^d;((X^q)^\ast)^{\frac{1}{q}})$ we have $\|g\|_{((X^q)^\ast)^{\frac{1}{q}}}\in L^\infty_c(\R^d)$, applying \ref{it:vvsdeq2} with $g$ replaced by $\|g\|_{((X^q)^\ast)^{\frac{1}{q}}}$ proves \ref{it:vvsdeq1}.

For \ref{it:vvsdeq1}$\Rightarrow$\ref{it:vvsdeq2} we note that by duality (see e.g. \cite[Proposition~1.3.1]{HNVW16}) we have
\begin{equation}\label{eq:vvsdequiv1}
\begin{split}
\nrmb{\|\widetilde{T}(\vec{f}\hspace{2pt})\|_X\cdot g}_{L^q(\R^d)}&=\nrmb{\||\widetilde{T}(\vec{f}\hspace{2pt})|^q\|_{X^q}\cdot |g|^q}^{\frac{1}{q}}_{L^1(\R^d)}\\&=\nrmb{|\widetilde{T}(\vec{f}\hspace{2pt})|^q|g|^q|}_{L^1(\R^d;X^q)}^{\frac{1}{q}}\\
&=\sup_{\|h\|_{L^\infty(\R^d;((X^q)^\ast)^{1/q})}=1}\nrmb{|\widetilde{T}(\vec{f}\hspace{2pt})|^q\cdot|g|^q \cdot |h|^q}_{L^1(\R^d;L^1(\Omega))}^{\frac1q}\\
&=\sup_{\|h\|_{L^\infty(\R^d;((X^q)^\ast)^{1/q})}=1}\nrmb{\widetilde{T}(\vec{f}\hspace{2pt})\cdot gh}_{L^q(\R^d;L^q(\Omega))}.
\end{split}
\end{equation}
Since $gh\in L^\infty_c(\R^d;((X^q)^\ast)^{\frac{1}{q}})$ for any $g\in L^\infty_c(\R^d)$ and $h\in L^\infty(\R^d;((X^q)^\ast)^{\frac{1}{q}})$ of norm 1 with $\|gh\|_{((X^q)^\ast)^{\frac{1}{q}}}\leq|g|\|h\|_{L^\infty(\R^d;((X^q)^\ast)^{\frac{1}{q}})}=|g|$, it follows from \ref{it:vvsdeq1} that
\begin{align*}
\nrmb{\widetilde{T}(\vec{f}\hspace{2pt})\cdot gh}_{L^q(\R^d;L^q(\Omega))}&\leq C\,\nrmb{M_{(\vec{r},\frac{1}{\frac{1}{q}-\frac{1}{s}})}(\|\vec{f}\|_{\vec{X}}, \|gh\|_{((X^q)^\ast)^{\frac{1}{q}}})}_{L^q(\R^d)}\\
&\leq C\,\nrmb{M_{(\vec{r},\frac{1}{\frac{1}{q}-\frac{1}{s}})}(\|\vec{f}\|_{\vec{X}}, g)}_{L^q(\R^d)}.
\end{align*}
By combining this result with \eqref{eq:vvsdequiv1} we have proven \ref{it:vvsdeq2}.
\end{proof}

In the following result we will deduce weighted bounds from domination by the multisublinear Hardy--Littlewood operator. To this end we introduce some terminology.
\begin{definition}
Let $\vec{X}$ be an $m$-tuple of quasi-Banach function spaces over a measure space $(\Omega,\mu)$. Let $\vec{Y}$, $V$ be $m+1$ quasi-normed linear subspaces of $L^0(\R^n;\vec{X})$ and $L^0(\R^n;X)$ respectively and let $T:\vec{Y}\to V$. We say that $T$ is \textit{$m$-linear} if it is linear in each of its components. We say that $T$ is \textit{$m$-sublinear} if it is positive-valued and subadditive in each of its components, i.e., the function $T(\vec{f}\hspace{2pt})$ takes values in the positive functions, and for all $j\in\{1,\ldots,m\}$
\[
T(f_1,\ldots,f_{j-1},f_j+g_j,f_{j+1},\ldots,f_m)\leq T(\vec{f})+T(f_1,\ldots,f_{j-1},g_j,f_{j+1},\ldots,f_m)
\]
for all $\vec{f}\in\vec{Y}$, $g_j\in Y_j$.
\end{definition}
We will generally consider operators that are either $m$-linear or $m$-sublinear, which we shorten by saying that the operator is $m$-(sub)linear. We point out that the modulus of any $m$-linear operator is $m$-sublinear.
\begin{proposition}\label{prop:wfromsdscalar}
Let $\vec{r}\in(0,\infty)^m$, $q\in(0,\infty)$, $s\in(q,\infty]$ and let $\vec{X}$ be an $m$-tuple of quasi-Banach function spaces over a measure space $(\Omega,\mu)$ such that $X$ is $q$-convex and order-continuous. Let $\widetilde{T}$ be an $m$-(sub)linear operator initially defined for all simple functions $\vec{f}\in L^\infty_c(\R^d;\vec{X})$. Suppose that
\begin{equation}\label{eq:scalarwpropin}
\nrmb{\|\widetilde{T}(\vec{f}\hspace{2pt})\|_X\cdot g}_{L^q(\R^d)}\leq C_T\,\nrmb{M_{(\vec{r},\frac{1}{\frac{1}{q}-\frac{1}{s}})}(\|\vec{f}\|_{\vec{X}}, g)}_{L^q(\R^d)}.
\end{equation}
for all simple $\vec{f}\in L^\infty_c(\R^d;\vec{X})$, $g\in L^\infty_c(\R^d)$. Then for all $\vec{p}\in (0,\infty]^m$ with $\vec{r}<\vec{p}$ and $p<s$, all $\vec{w}\in A_{\vec{p},(\vec{r},s)}$, $\widetilde{T}$ has a unique extension satisfying
\[
\|\widetilde{T}(\vec{f}\hspace{2pt})\|_{L_w^p(\R^d;X)}\lesssim_{\vec{p},q,\vec{r},s} C_T\, [\vec{w}]_{\vec{p},(\vec{r},s)}^{ \max\left\{\frac{\frac{1}{\vec{r}}}{\frac{1}{\vec{r}}-\frac{1}{\vec{p}}},\frac{\frac{1}{q}-\frac{1}{s}}{\frac{1}{p}-\frac{1}{s}}\right\}}\prod_{j=1}^m\|f_j\|_{L^{p_j}_{w_j}(\R^d;X_j)}
\]
for all $\vec{f}\in L^{\vec{p}}_{\vec{w}}(\R^d;\vec{X})$.
\end{proposition}

Proposition~\ref{prop:wfromsdscalar} is essentially a consequence of Proposition~\ref{prop:multimaxwest} and, in certain cases, the quantitative multilinear extrapolation result in \cite{Ni19}. The reason we might have to use extrapolation is because sparse domination by forms yields, a priori, weighted bounds for the range of exponents where one can dualize the operator. Typically, in the multilinear case, this does not yield the full range of exponents where the operator satisfies weighted bounds. To recover this full range of exponents, we will use the following version of the extrapolation theorem in \cite{Ni19}:

\begin{theorem}\label{thm:qextrapolation}
Let $\vec{r}\in(0,\infty)^m$, $q \in (0,\infty)$, $s\in(q,\infty]$ and let $\vec{t}\in(0,\infty]^m$ satisfy $\vec{t}\geq\vec{r}$ and $q\leq t\leq s$.
Suppose we are given $\vec{p}\in(0,\infty]^m$ satisfying $\vec{p}>\vec{r}$, $q\leq p<s$, $\vec{w}\in A_{\vec{p},(\vec{r},s)}$, and $\vec{f}\in L^{\vec{p}}_{\vec{w}}(\R^n)$, $g\in L^{\frac{1}{\frac{1}{q}-\frac{1}{p}}}_{w^{-1}}(\R^n)$. Then there is a $\vec{W}\in A_{\vec{t},(\vec{r},s)}$ such that
\begin{equation}\label{eq:mainprop11}
\has{\prod_{j=1}^m\|f_j\|_{L^{t_j}_{W_j}(\R^n)}}\|g\|_{L^{\frac{1}{\frac{1}{q}-\frac{1}{t}}}_{W^{-1}}(\R^n)}\leq 2^{\frac{m^2}{q}}\has{\prod_{j=1}^m\|f_j\|_{L^{p_j}_{w_j}(\R^n)}}\|g\|_{L^{\frac{1}{\frac{1}{q}-\frac{1}{p}}}_{w^{-1}}(\R^n)}
\end{equation}
and
\begin{equation}\label{eq:mainprop22}
[\vec{W}]_{\vec{t},(\vec{r},s)}\lesssim_{\vec{p},\vec{r},s,\vec{t}}[\vec{w}]^{ \max\left\{\frac{\frac{1}{\vec{r}}-\frac{1}{\vec{t}}}{\frac{1}{\vec{r}}-\frac{1}{\vec{p}}},\frac{\frac{1}{t}-\frac{1}{s}}{\frac{1}{p}-\frac{1}{s}}\right\}}_{\vec{p},(\vec{r},s)}.
\end{equation}
\end{theorem}
\begin{proof}
This follows from an application of \cite[Theorem~3.1]{Ni19} with $\pmb{r}=(\frac{\vec{r}}{q},(\frac{s}{q})')$, $\pmb{q}=(\frac{\vec{t}}{q},(\frac{t}{q})')$, $\pmb{p}=(\frac{\vec{p}}{q},(\frac{p}{q})')$, $\pmb{w}=(\vec{w}^q,w^{-q})$, and $\pmb{f}=(|\vec{f}|^q,|g|^q)$.
\end{proof}

Next, we prove an extension lemma for multi(sub)linear operators, which will be needed in the proof of Proposition~\ref{prop:wfromsdscalar}. In the case $m=1$, a bounded (sub)linear operator satisfies a reverse triangle inequality type estimate and thus, in particular, is uniformly continuous. Therefore, if it takes values in a complete space, it extends to an operator on the closure of its domain. For $m>2$ this uniform continuity needs to be replaced by a local uniform continuity. This again suffices to extend the operator to the closure of its domain. While this result is straightforward, we include it here for convenience of the reader.
\begin{lemma}\label{lem:extlem}
Let $\vec{Y}$ be an $m$-tuple of quasi-normed vector spaces, let $Z$ be a quasi-Banach function space,  and let $U_j\subseteq Y_j$ be a dense subspace for each $j\in\{1,\ldots,m\}$. If $T:\vec{U}\to Z$
is bounded and satisfies the pointwise a.e. estimate
\begin{equation}\label{eq:extlemest}
\begin{split}
|T(\vec{f}\hspace{2pt})-T(\vec{g}\hspace{2pt})|&\leq\sum_{j=1}^m|T(f_1,\ldots,f_{j-1},f_j-g_j,g_{j+1},\ldots,g_m)|\\ &\quad+|T(g_1,\ldots,g_{j-1},g_j-f_j,f_{j+1},\ldots,f_m)|
\end{split}
\end{equation}
for all $\vec{f},\vec{g}\in\vec{U}$, then $T$ uniquely extends to a bounded operator $\vec{Y}\to Z$ with a comparable bound.

If $T$ is $m$-(sub)linear, then it satisfies \eqref{eq:extlemest} and its extension, again denoted by $T$, is an $m$-(sub)linear operator as well.
\end{lemma}

\begin{proof}
By the compatibility of the norm on $Z$ with pointwise estimates, \eqref{eq:extlemest} and boundedness of $T$ yields for $\vec{f},\vec{g}\in\vec{U}$
\begin{equation}\label{eq:extlemest2}
\begin{split}
\|T(\vec{f}\hspace{2pt})-T(\vec{g}\hspace{2pt})\|_Z&\lesssim\sum_{j=1}^m\|T(f_1,\ldots,f_{j-1},f_j-g_j,g_{j+1},\ldots,g_m)\|_Z\\ &\quad+\|T(g_1,\ldots,g_{j-1},g_j-f_j,f_{j+1},\ldots,f_m)\|_Z\\
&\lesssim\sum_{j=1}^m\has{\prod_{\substack{l=1\\l\neq j}}^m(\|f_l\|_{Y_l}+\|g_l\|_{Y_l})}\|f_j-g_j\|_{Y_j}.
\end{split}
\end{equation}
Now, if $\vec{f}\in\vec{Y}$ and $(f_j^k)_{k\in\N}$ is a sequence in $U_j$ converging to $f_j$ in $Y_j$ for all $j\in\{1,\ldots,m\}$, then \eqref{eq:extlemest2} implies that $(T(\vec{f}\hspace{2pt}^k))_{k\in\N}$ is a Cauchy-sequence in $Z$. The first assertion then follows by defining $T(\vec{f}\hspace{2pt})$ to be the limit of this sequence in $Z$. Note that this is well-defined since it follows from another application of \eqref{eq:extlemest2} that this limit does not depend on the approximating sequences of the $f_j$. For the bound we have
\[
\|T(\vec{f}\hspace{2pt})\|_Z\leq\beta\liminf_{k\to\infty}\|T(\vec{f}\hspace{2pt}^k)\|_Z\leq\beta c\prod_{j=1}^m\limsup_{k\to\infty}\|f_j^k\|_{Y_j}\leq\beta\has{\prod_{j=1}^m\alpha_j}c \prod_{j=1}^m\|f_j\|_{Y_j}.
\]
where $c$, $\alpha_j$, and $\beta$ are respectively the bound for $T$, the quasi-triangle inequality constant of $Y_j$, and the quasi-triangle inequality constant of $Z$.

If $T$ is $m$-sublinear, then it follows from iterating the inequality
\[
T(\vec{f}\hspace{2pt})\leq T(g_1,f_2,\ldots,f_m)+T(f_1-g_1,f_2,\ldots,f_m)
\]
for all $f_j$ in the first term on the right for $j=2$ to $j=m$, that
\[
T(\vec{f}\hspace{2pt}) \leq T(\vec{g}\hspace{2pt})+\sum_{j=1}^mT(f_1,\ldots,f_{j-1},f_j-g_j,g_{j+1},\ldots,g_m).
\]
By symmetry, we obtain
\[
T(\vec{g}\hspace{2pt})\leq T(\vec{f}\hspace{2pt})+\sum_{j=1}^mT(g_1,\ldots,g_{j-1},g_j-f_j,f_{j+1},\ldots,f_m)
\]
and by combining these two estimates we obtain \eqref{eq:extlemest}. If $T$ is $m$-linear, these first two inequalities are actually equalities, so we can proceed analogously.
The final assertion is a consequence of the fact that $m$-(sub)linearity is a pointwise property and that convergence in $Z$ implies local convergence in measure (see e.g. \cite[Theorem 1.4]{BS88}), and thus a.e. convergence on a subsequence.
\end{proof}
Note that if $Z=L^p_w(\R^n)$ for a weight $w$ and a $p\in(0,\infty]$, then by continuity of the (quasi-)norm $\|\cdot\|_{L^p_w(\R^n)}$, the extension $T$ will have the same bound as the original $T$.

Since we are working with weights that are not necessarily locally integrable, it is not a-priori clear that the simple functions of compact support are dense in the weighted Lebesgue spaces. We prove that this density result does indeed hold.
\begin{lemma}\label{lem:dommultdense}
Let $w$ be a weight, $p,q\in(0,\infty)$ and $X$ a $q$-convex quasi-Banach function space. Then the simple functions in $L^p_w(\R^n;X)\cap L^\infty_c(\R^n;X)$ are dense in $L^p_w(\R^n;X)$.
\end{lemma}

\begin{proof}
  First suppose that $p,q \geq 1$ and fix $f \in L^p_w(\R^n;X)$. By \cite[Corollary 1.1.21]{HNVW16} and the dominated convergence theorem there exists a sequence of simple functions $(f_k)_{k \in \N}$  such that $f_k \to f$ in $L^p_w(\R^d;X)$, and $f_k(x) \to f(x)$ and $\nrm{f_k(x)}_X \leq \nrm{f}_X$ for a.e. $x \in \R^d$. Setting $(g_k)_{k\in \N} = (f_k \ind_{B(0,k)})_{k \in \N}$ it follows that $g_k \in L^p_w(\R^n;X)\cap L^\infty_c(\R^n;X)$ for all $k \in \N$ and
  $g_k \to f$ in $L^p_w(\R^n;X)$ by the dominated convergence theorem, proving the lemma.

Now consider the case $p< 1$ and/or $q < 1$. Fix $n\in\N$ so that $2^np,2^nq>1$. For $f\in L^p_w(\R^n;X)$ we can pick a positive $g\in L^{2^np}_{w^{2^{-n}}}(\R^n;X^{2^{-n}})$ with $g^{2^n}=|f|$. By our previous result we can find a positive sequence of simple functions $(g_k)_{k\in\N}$ in $L^{2^np}_{w^{2^{-n}}}(\R^n;X^{2^{-n}})\cap L^{\infty}_c(\R^n;X^{2^{-n}})$ converging to $g$. Setting $f_k:=g_k^{2^n}\sgn(f)\in L^p_w(\R^n;X)\cap L^\infty_c(\R^n;X)$ we compute
\[
|f_k-f|=|g_k^{2^n}-g^{2^n}|=|g_k-g|\prod_{l=0}^{n-1}|g_k^{2^l}+g^{2^l}|
\]
so that by H\"older's inequality
\[
\|f_k-f\|_{L^p_w(\R^n;X)}\leq\|g_k-g\|_{L^{2^np}_{w^{2^{-n}}}(\R^n;X^{2^{-n}})} \prod_{l=0}^{n-1}\|g_k^{2^l}+g^{2^l}\|_{L^{2^{n-l}p}_{w^{2^{-(n-l)}}}(\R^n;X^{2^{-n-l}})}.
\]
Since $\|g_k^{2^l}+g^{2^l}\|$ is bounded in $k$, we conclude that $f_k\to f$ in $L^p_w(\R^n;X)$.
\end{proof}

We are now ready to prove Proposition~\ref{prop:wfromsdscalar}.
\begin{proof}[Proof of Proposition~\ref{prop:wfromsdscalar}]
Set $\frac{1}{t_j}:=\frac{\tau}{r_j}$ with $\frac{1}{\tau}=\frac{\frac{1}{r}-\frac{1}{s}+\frac{1}{q}}{\frac{1}{q}}>1$. Noting that $$\frac{1}{t}=\frac{\tau}{r}=\frac{\frac{1}{r}}{\frac{1}{r}+\frac{1}{q}-\frac{1}{s}} \frac{1}{q}\in(\frac{1}{s},\frac{1}{q}),$$
\[
\frac{\frac{1}{r_1}}{\frac{1}{r_1}-\frac{1}{t_1}}=\ldots=\frac{\frac{1}{r_m}}{\frac{1}{r_m}-\frac{1}{t_m}} =\frac{\frac{1}{q}-\frac{1}{s}}{\frac{1}{t}-\frac{1}{s}}=\frac{1}{1-\tau},
\]
and
\[
\bracb{(\vec{W},W^{-1})}_{(\vec{t},\frac{1}{\frac{1}{q}-\frac{1}{t}}),((\vec{r},\frac{1}{\frac{1}{q}-\frac{1}{s}}),\infty)}=[\vec{W}]_{\vec{t},(\vec{r},s)},
\]
for $\vec{W}\in A_{\vec{t},(\vec{r},s)}$, it follows from \eqref{eq:scalarwpropin} and Proposition~\ref{prop:multimaxwest} that
\begin{equation}\label{eq:wfromsdscalar1}
\begin{split}
\nrmb{\|\widetilde{T}(\vec{f}\hspace{2pt})\|_X\cdot g}_{L^q(\R^d)}&\leq C_T\,\nrmb{M_{(\vec{r},\frac{1}{\frac{1}{q}-\frac{1}{s}})}(\|\vec{f}\|_{\vec{X}}, g)}_{L^q(\R^d)}\\
&\lesssim_{\vec{t},\vec{r}} C_T\, [\vec{W}]^{\frac{1}{1-\tau}}_{\vec{t},(\vec{r},s)} \Big(\prod_{j=1}^m\|f_j\|_{L^{t_j}_{W_j}(\R^d;X_j)}\Big)\|g\|_{L_{W^{-1}}^{\ha{\frac{1}{q}-\frac{1}{t}}^{-1}}(\R^d)}
\end{split}
\end{equation}
for all $\vec{W}\in A_{\vec{t},(\vec{r},s)}$ and all simple $f_j\in L^{t_j}_{W_j}(\R^d)\cap L^\infty_c(\R^n)$, $g\in L_{W^{-1}}^{\ha{\frac{1}{q}-\frac{1}{t}}^{-1}}(\R^d)\cap L^\infty_c(\R^n)$. By Lemma~\ref{lem:extlem} and Lemma~\ref{lem:dommultdense}, $\widetilde{T}$ has a unique extension satisfying this estimate for all $\vec{W}\in A_{\vec{t},(\vec{r},s)}$, $\vec{f}\in L^{\vec{t}}_{\vec{W}}(\R^d;\vec{X})$, and $g\in L_{W^{-1}}^{\ha{\frac{1}{q}-\frac{1}{t}}^{-1}}(\R^d)$.

Now let $\vec{p}\in (0,\infty]^m$ with $\vec{r}<\vec{p}$ and $p<s$, $\vec{w}\in A_{\vec{p},(\vec{r},s)}$, and $\vec{f}\in L^{\vec{p}}_{\vec{w}}(\R^n;\vec{X})$, $g\in L_{w^{-1}}^{\ha{\frac{1}{q}-\frac{1}{p}}^{-1}}(\R^d)$. It follows from applying Theorem~\ref{thm:qextrapolation} with $f_j$ replaced by $\|f_j\|_{X_j}$, that there is a $\vec{W}\in A_{\vec{t},(\vec{r},s)}$ such that
\[
\has{\prod_{j=1}^m\|f_j\|_{L^{t_j}_{W_j}(\R^n;X_j)}}\|g\|_{L^{\frac{1}{\frac{1}{q}-\frac{1}{t}}}_{W^{-1}}(\R^n)}\leq 2^{\frac{m^2}{q}}\has{\prod_{j=1}^m\|f_j\|_{L^{p_j}_{w_j}(\R^n;X_j)}}\|g\|_{L^{\frac{1}{\frac{1}{q}-\frac{1}{p}}}_{w^{-1}}(\R^n)}
\]
and
\[
[\vec{W}]_{\vec{t},(\vec{r},s)}\lesssim_{\vec{p},\vec{r},s,\vec{t}}[\vec{w}]^{(1-\tau)\cdot \max\left\{\frac{\frac{1}{\vec{r}}}{\frac{1}{\vec{r}}-\frac{1}{\vec{p}}},\frac{\frac{1}{q}-\frac{1}{s}}{\frac{1}{p}-\frac{1}{s}}\right\}}_{\vec{p},(\vec{r},s)}.
\]
Then it follows from \eqref{eq:wfromsdscalar1} that
\begin{align*}
\nrmb{\|\widetilde{T}(\vec{f}\hspace{2pt})\|_X\cdot g}_{L^q(\R^d)}
&\lesssim_{\vec{t},\vec{r}} C_T\, [\vec{W}]^{\frac{1}{1-\tau}}_{\vec{t},(\vec{r},s)} \Big(\prod_{j=1}^m\|f_j\|_{L^{t_j}_{W_j}(\R^d;X_j)}\Big)\|g\|_{L_{W^{-1}}^{\ha{\frac{1}{q}-\frac{1}{t}}^{-1}}(\R^d)}\\
&\lesssim_{\vec{p},q,\vec{r},s,\vec{t}}C_T[\vec{w}]^{ \max\left\{\frac{\frac{1}{\vec{r}}}{\frac{1}{\vec{r}}-\frac{1}{\vec{p}}},\frac{\frac{1}{q}-\frac{1}{s}}{\frac{1}{p}-\frac{1}{s}}\right\}}_{\vec{p},(\vec{r},s)}
\Big(\prod_{j=1}^m\|f_j\|_{L^{p_j}_{w_j}(\R^d;X_j)}\Big)\|g\|_{L_{w^{-1}}^{\ha{\frac{1}{q}-\frac{1}{p}}^{-1}}(\R^d)}.
\end{align*}
The assertion now follows from the duality
\[
\nrmb{T(\vec{f}\hspace{2pt})}_{L^p_w(\R^d;X)}=\nrmb{\|T(\vec{f}\hspace{2pt})\|_X^q}^{\frac{1}{q}}_{L^{\frac{p}{q}}_{w^q}(\R^d)} =\sup_{\substack{\|g\|_{L_{w^{-1}}^{\ha{\frac{1}{q}-\frac{1}{p}}^{-1}}\hspace{-10pt}(\R^d)}=1}}\nrmb{\|\widetilde{T}(\vec{f}\hspace{2pt})\|_X\cdot g}_{L^q(\R^d)}.
\]
\end{proof}

\section{The multisublinear lattice maximal operator}\label{section:multiHL}
In this section we will introduce and study properties of the multilinear lattice maximal operator, which will play a major role in the proof of our main theorem.
 We will start by reviewing the case $m=1$. Let $X$ be a Banach function space and let $\mc{D}$ be a finite collection of cubes in $\R^d$. For any $f\in L^1_{\loc}(\R^d;X)$ we define
  \begin{equation*}
   \widetilde{M}^{\mc{D}}f(x):= \sup_{Q\in \mc{D}} \, \ipb{f}_{1,Q} \ind_Q(x), \qquad x\in \R^d.
  \end{equation*}
  where the supremum is taken in the lattice sense. We say that $X$ has the \emph{Hardy--Littlewood property} and write $X \in \HL$ if for some $p \in (1,\infty)$
  \begin{equation*}
     \nrm{\widetilde{M}}_{p,X}:= \sup_{\mc{D}}\, \nrmb{\widetilde{M}^{\mc{D}}}_{L^p(\R^d;X) \to L^p(\R^d;X)} <\infty,
  \end{equation*}
  where the supremum is taken over all  finite collection of cubes $\mc{D}$. This property is independent of the exponent $p$ and the dimension $d$ (see \cite{GMT93})
  and even the operator norm $ \nrm{\widetilde{M}}_{p,X}$ can be bounded by a constant independent of the dimension $d$ (see \cite{DK17b}).

  As an example we note that (iterated) $L^p$-spaces for $p \in (1,\infty]$ have the Hardy--Littlewood property. Moreover by a deep result of Bourgain \cite{Bo84} and Rubio de Francia \cite[Theorem 3]{Ru86} we have that both $X$ and $X^*$ have the Hardy--Littlewood property if and only if $X$ has the so-called $\UMD$ property. We will elaborate on the connection between the Hardy--Littlewood property and the $\UMD$ property in Section \ref{section:multiUMD}.

  If $X$ is an order-continuous Banach function space with the Hardy--Littlewood property and $p \in [1,\infty)$, we  define the \emph{lattice Hardy--Littlewood maximal operator} for $f \in L^p(\R^d;X)$ by
    \begin{equation*}
    \widetilde{M}f(x):= \sup_{Q} \, \ipb{f}_{1,Q} \ind_Q(x), \qquad x\in \R^d.
  \end{equation*}
  where the supremum is taken in the lattice sense over all cubes $Q \subseteq \R^d$. To see that
 $\widetilde{M}f:\R^d \to X$ is strongly measurable, let $\mc{D}_n$ be a finite collection of cubes for each $n \in \N$ such that $\mc{D}_n \subseteq \mc{D}_{n+1}$ and such that
 \begin{equation*}
    \widetilde{M}f(x)= \sup_{n \in \N}\widetilde{M}^{\mc{D}_n}f(x), \qquad x\in \R^d.
  \end{equation*}
By the Hardy--Littlewood property of $X$ we know that
$\sup_{n \in \N} \nrm{\widetilde{M}^{\mc{D}_n}f}_{L^{p,\infty}(\R^d;X)} <\infty$
 (see \cite[Theorem 1.7]{GMT93}).  Thus using the Fatou property of $X$ it follows  that $\widetilde{M}f(x) \in X$   for a.e. $x \in \R^d$. Moreover since $X$ is order-continuous, $\hab{\widetilde{M}^{\mc{D}_n}f(x)}_{n \in \N}$ converges to $\widetilde{M}f(x)$ for a.e. $x \in \R^d$. As $\widetilde{M}^{\mc{D}_n}f$ is a simple function for each $n\in \N$, we can conclude that $\widetilde{M}f$ is strongly measurable, i.e. $\widetilde{M}f \in L^0(\R^d;X)$.

\bigskip

Let us now turn to the multisublinear, rescaled generalization of the lattice Hardy--Littlewood maximal operator that we will need for our main result. Take $\vec{r} \in (0,\infty)^{m}$ and let $\vec{X}$ be an $\vec{r}$-convex $m$-tuple of quasi-Banach function spaces. For $\vec{f} \in L^{\vec{r}}_{\loc}(\R^d;\vec{X})$ and a finite collection of cubes $\mc{D}$ in $\R^d$ we define the multilinear analog of $\widetilde{M}^{\mc{D}}$ as
  \begin{equation*}
    \widetilde{M}_{\vec{r}}^{\mc{D}}(\vec{f}\hspace{2pt})(x):= \sup_{Q\in \mc{D}} \, \prod_{j=1}^{m}\ipb{f_j}_{r_j,Q} \ind_Q(x), \qquad x\in \R^d,
  \end{equation*}
where the supremum is taken in the lattice sense. Note that  $\widetilde{M}_{(1)}^{\mc{D}}=\widetilde{M}^{\mc{D}}$.

  \begin{definition}
   Let $\vec{X}$ be an $m$-tuple of quasi-Banach function spaces and take $\vec{r} \in (0,\infty)^{m}$.
   We say that $\vec{X}$ has the \emph{$\vec{r}$-Hardy--Littlewood property} and write $\vec{X} \in \HL_{\vec{r}}$ if $\vec{X}$ is $\vec{r}$-convex and for some $\vec{r} < \vec{p} \leq\infty$ we have
  \begin{equation*}
    \nrm{\widetilde{M}_{\vec{r}}}_{\vec{p},\vec{X}}:= \sup_{\mc{D}}\, \nrmb{\widetilde{M}_{\vec{r}}^{\mc{D}}}_{L^{\vec{p}}(\R;\vec{X}) \to L^p(\R;X)}<\infty,
  \end{equation*}
  where the supremum is taken over all finite collection of cubes.
  \end{definition}

 As in the linear case, the definition of $\HL_{\vec{r}}$ is independent of the exponent $p$ and the dimension $d$. The independence of $d$ can be shown using the method of rotations (see e.g. \cite[Remark 1.3]{GMT93}), and the independence of $p$ will follow from Corollary \ref{corollary:weightedHL} below.

  The multilinear Hardy--Littlewood property has the following properties:

\begin{proposition}\label{prop:HLlintomulti}
Let $\vec{X}$ be an $m$-tuple of quasi-Banach function spaces and take $\vec{r} \in (0,\infty)^{m}$.
If $X_j^{r_j} \in \HL$ for $1\leq j\leq m$, then $\vec{X} \in \HL_{\vec{r}}$.
\end{proposition}
\begin{proof}
Fix a finite collection of dyadic cubes $\mc{D}$ . Let
 $r_j<p_j\leq \infty$ be such that
$\nrm{\widetilde{M}}_{p_j/r_j,X_j^{r_j}}<\infty$ for
$1\leq j\leq m$.
For $\vec{f} \in L^{\vec{p}}(\R^d,\vec{X})$  we have by H\"older's inequality
\begin{align*}
\nrmb{\widetilde{M}_{\vec{r}}^\mc{D}\vec{f}\hspace{2pt}}_{L^p(\R^d;X)}
\leq \prod_{j=1}^m\nrmb{\widetilde{M}^\mc{D}(\abs{f_j}^{r_j})}_{L^{{p_j}/{r_j}}(\R^d;X_j^{r_j})}^{1/r_j}\leq \prod_{j=1}^m \nrm{\widetilde{M}}_{p_j/r_j,X_j^{r_j}}^{1/r_j} \nrm{f_j}_{L^{p_j}(\R^d;X_j)}.
\end{align*}
Thus taking the supremum over all $\vec{f}$ of norm $1$ and all finite collection of cubes $\mc{D}$ yields $\vec{X} \in \HL_{\vec{r}}$.
\end{proof}

We point out that Proposition \ref{prop:HLlintomulti} does not provide a necessary condition. Indeed, for $m=3$ we can take $X_1=\ell^2(\ell^\infty)$, $X_2=\ell^\infty(\ell^2)$ and $X_3=\ell^2(\ell^2)$. It is shown in \cite[Proposition 8.1]{NVW15} that $X_2$ does not satisfy the Hardy-Littlewood property. However, noting that $X_3=(X_1\cdot X_2)^\ast$, it follows from Proposition~\ref{prop:lebesgueumd} below that $\vec{X}\in\HL_{(1,1,1)}$.

\bigskip

Take $\vec{r},\vec{p} \in (0,\infty]^m$ with $\vec{r} \leq \vec{p}$, assume that $\vec{X} \in \HL_{\vec{r}}$ and that $X$ is order-continuous. We  define the \emph{multisublinear lattice maximal operator}
    \begin{equation*}
    \widetilde{M}_{\vec{r}}(\vec{f}\hspace{2pt})(x):= \sup_{Q} \, \prod_{j=1}^{m}\ipb{f_j}_{r_j,Q} \ind_Q(x), \qquad x\in \R^d,
  \end{equation*}
  for $\vec{f}\in L^{\vec{p}}(\R^d;\vec{X})$, where the supremum is taken in the lattice sense over all cubes $Q \subseteq \R^d$. As in the case $m=1$, the order-continuity of $X$ ensures that $\widetilde{M}_{\vec{r}}(\vec{f}\hspace{2pt}) \in L^0(\R^d;X)$.

We will show a sparse domination result for $\widetilde{M}_{\vec{r}}$, which is key in the proof of our main result. Sparse domination in the case $m=1$ was studied in \cite{HL17}. Here we will adapt the arguments from \cite{HL17} to the case where $m>1$. As a first step we will show a weak endpoint for $\widetilde{M}_{\vec{r}}^{\mc{D}}$.
\begin{lemma}\label{lemma:weakbiHL}
  Let $\vec{X}$ be an $m$-tuple of quasi-Banach function spaces, take $\vec{r} \in (0,\infty)^{m}$ and suppose $\vec{X} \in \HL_{\vec{r}}$. Then we have for $\vec{r}<\vec{p}\leq \infty$
  \begin{equation*}
    \sup_{\mc{D}\subseteq \ms{D} \text{ finite}}\,\nrmb{\widetilde{M}^{\mc{D}}_{\vec{r}}}_{L^{\vec{r}}(\R^d;\vec{X})  \to L^{r,\infty}(\R^d;X) } \lesssim_{\vec{p},\vec{r}} \nrm{\widetilde{M}_{\vec{r}}}_{\vec{p},X}
  \end{equation*}
\end{lemma}
\begin{proof}
Fix $\mc{D}\subseteq \ms{D}$ finite and take $\vec{f} \in L^{\vec{r}}(\R^d;\vec{X})$. By scaling we may assume that $\nrm{f_j}_{L^{r_j}(\R^d;X_j)}=1$ for $1\leq j\leq m$. For $\lambda >0$ and $1\leq j\leq m$ define
\begin{equation*}
  \mc{S}_j := \cbraceb{Q \in \ms{D}: Q \text{ maximal (w.r.t inclusion) such that }\ipb{\nrm{f_j}_{X_j}}_{r_j,Q} >\lambda^{r/r_j}}
\end{equation*}
and set $O_j := \bigcup_{Q \in \mc{S}_j}Q=\cbrace{M_{r_j}^{\ms{D}}(\nrm{f_j}_{X_j})>\lambda^{r/r_j}}$. For a fixed $P \in \mc{D}$ and $1\leq j\leq m$ we have if $P\setminus O_j\neq \emptyset$, then
\begin{align*}
  \ip{f_j}_{r_j,P}\ind_P &=  \ips{f_j \ind_{O^{\comp}_j} +\sum_{\substack{Q \in \mc{S}_j:\\Q\subseteq P}}f_j\ind_Q}_{r_j,P} \ind_P \\
  &= \ips{f_j \ind_{O^{\comp}_j}+\sum_{Q \in \mc{S}_j}\ip{f_j}_{r_j,Q}\ind_{Q}}_{r_j,P}\ind_P
\end{align*}
using the disjointness of the cubes in $\mc{S}_j$ and
 $$\ipb{\ip{f_j}_{r_j,Q}\ind_Q}_{r_j,P} = \ip{f_j\ind_Q}_{r_j,P}, \qquad Q\subseteq P$$
  in the second equality. Taking the product over $1\leq j \leq m$ and the supremum over $P \in \mc{D}$ we can estimate
  \begin{align*}
    \widetilde{M}^{\mc{D}}_{\vec{r}}(\vec{f}\hspace{2pt}) &\leq \sup_{P\in \mc{D}} \prod_{j=1}^m\has{ \ips{f_j \ind_{O^{\comp}_j}+\sum_{Q \in \mc{S}_j}\ip{f_j}_{r_j,Q}\ind_{Q}}_{r_j,P}\ind_P + \ip{f_j}_{r_j,P} \ind_{O_j}}\\
  &\leq \widetilde{M}^{\mc{D}}_{\vec{r}}(\vec{g}\hspace{2pt}) + b
\end{align*}
where
\begin{equation*}
  g_j := g_j^1+g_j^2:= f_j \ind_{O^{\comp}_j} + \sum_{Q \in \mc{S}_j}\ip{f_j}_{r_j,Q}\ind_{Q}, \qquad 1\leq j \leq m.
\end{equation*}
and $b:\R^d \to X$ is the sum of all terms of the product over $1\leq j \leq m$ except $\widetilde{M}^{\mc{D}}_{\vec{r}}(\vec{g})$. By the disjointness of the cubes in $\mc{S}_j$ we have $\nrm{g_j}_{L^{r_j}(\R^d;X_j)} =\nrm{f_j}_{L^{r_j}(\R^d;X_j)} =1$. Moreover since
$$\supp b \subseteq \bigcup_{j=1}^m O_j = \bigcup_{j=1}^m \cbraceb{M_{r_j}^{\ms{D}}(\nrm{f_j}_{X_j})>\lambda^{r/r_j}}$$
and since $M_{r_j}^{\ms{D}}$ is weak $L^{r_j}$-bounded
we have
\begin{equation*}
  \absb{\nrm{b}_X>\lambda} \leq \sum_{j=1}^m \absb{M_{r_j}^{\ms{D}}\hab{\nrm{f_j}_{X_j}}>\lambda^{r/r_j}} \leq \frac{m}{\lambda^{r}}.
\end{equation*}
Next we estimate the $L^\infty$-norm of $\vec{g}$. For $1\leq j\leq m$ we have by the Lebesgue differentiation theorem that
\begin{equation*}
  \nrm{g_j^1}_{X_j} = \nrm{f_j}_{X_j} \ind_{O_j^{\comp}} \leq M_{r_j}^{\ms{D}}\hab{\nrm{f_j}_{X_j}}\ind_{O_j^{\comp}} \leq \lambda^{r/r_j}.
\end{equation*}
and, using the disjointness of the cubes in $\mc{S}_j$, the $r_j$-convexity of $X_j$, and the maximality of the cubes in $\mc{S}_j$, we have
\begin{equation*}
  \nrm{g_j^2}_{X_j} = \Big\|\Big(\sum_{Q \in \mc{S}_j}\ip{f_j}^{r_j}_{r_j,Q}\ind_{Q}\Big)^{\frac{1}{r_j}}\Big\|_{X_j}\leq 2^{d/r_j} \Big(\sum_{Q \in \mc{S}_j}\ip{\|f_j\|_{X_j}}^{r_j}_{r_j,\hat{Q}}\ind_{Q}\Big)^{\frac{1}{r_j}} \leq 2^{d/r_j} \lambda ^{r/r_j},
\end{equation*}
where $\hat{Q}$ is the dyadic parent of $Q \in \mc{S}_j$. Thus we have $\nrm{g_j}_{L^\infty(\R^d;X_j)} \lesssim_{r_j} \lambda^{r/r_j}$.

Combining the estimates for $\vec{g}$ and $b$ we obtain for $\vec{r} <\vec{p}< \infty$
\begin{align*}
  \abss{\nrmb{\widetilde{M}_{\vec{r}}^{\mc{D}} (\vec{f}\hspace{2pt})}_X > 2\lambda} &\leq \abss{\nrmb{\widetilde{M}_{\vec{r}}^{\mc{D}} (\vec{g})}_X > \lambda}+\absb{\nrm{b}_X > \lambda} \\&\leq \nrm{\widetilde{M}_{\vec{r}}}_{\vec{p},X} \, \frac{\prod_{j=1}^m\nrm{g_j}_{L^{p_j}(\R^d;X_j)}^{p}}{\lambda^{p}} + \frac{m}{\lambda^r} \\
  &\lesssim_{\vec{p},\vec{r}} \nrm{\widetilde{M}_{\vec{r}}}_{\vec{p},X}\, \frac{\prod_{j=1}^m\nrm{g_j}_{L^{r_j}(\R^d;X_j)}^{\frac{r_j}{p_j}p}\lambda^{(\frac{1}{r_j}-\frac{1}{p_j})pr}}{\lambda^{p}} + \frac{1}{\lambda^r}
  \leq\nrm{\widetilde{M}_{\vec{r}}}_{\vec{p},X}\, \frac{2}{\lambda^{r}}
\end{align*}
and the case where $p_j=\infty$ for some (or all) $1\leq j\leq m$ is similar.
Taking the supremum over $\vec{f}\in L^{\vec{r}}(\R^d;\vec{X})$ and all finite collections of cubes $\mc{D} \subseteq \ms{D}$ yields the conclusion.
\end{proof}

Using this lemma we now come to the main theorem of this section, which establishes sparse domination of the multisublinear lattice maximal operator. Recall that sparse domination is equivalent to domination by a multisublinear maximal operator, see Proposition~\ref{prop:ptwisetoform}, and recall Convention \ref{conv:prod} for the definition of $r$ and $X$.

\begin{theorem}\label{theorem:sparsedomHL}Let $\vec{X}$ be an $m$-tuple of quasi-Banach function spaces, take $\vec{r} \in (0,\infty)^{m}$ and $q \in [r,\infty)$. Suppose that $\vec{X} \in \HL_{\vec{r}}$ and that $X$ is an order-continuous $q$-convex quasi-Banach function space. Then for any $\vec{f} \in  L^{\vec{r}}_{\loc}(\R^d;\vec{X})$ and $g \in L^q_{\loc}(\R^d)$ we have
\begin{align*}
\nrmb{\nrm{\widetilde{M}_{\vec{r}}(\vec{f}\hspace{2pt})}_X\cdot g }_{L^q(\R^d)} &\lesssim_{\vec{X},\vec{r}}  \nrmb{M_{(\vec{r},q)}\hab{\nrm{\vec{f}}_{\vec{X}},g}}_{L^q(\R^d)},
\intertext{In particular, we have}
\nrmb{\widetilde{M}_{\vec{r}}(\vec{f}\hspace{2pt})}_{L^q(\R^d;X)} &\lesssim_{\vec{X},\vec{r}}  \nrmb{M_{\vec{r}}\hab{\nrm{\vec{f}}_{\vec{X}}}}_{L^q(\R^d)}.
\end{align*}
\end{theorem}

Note that $X$ is Theorem \ref{theorem:sparsedomHL} is automatically $r$-convex, since $X_j$ is $r_j$-convex for $1\leq j\leq m$. If $X$ is $q$-convex for $q>r$ we get a sparse domination result with a smaller sparse operator, which in turn yields better weighted bounds (see also Corollary \ref{corollary:weightedHL}).

\begin{proof}
We will first show that $\widetilde{M}_{\vec{r}}^{\mc{D}}$ is sparsely dominated for any finite collection of dyadic cubes $\mc{D}$ with an estimate independent of $\mc{D}$, from which we will deduce the claimed estimates for $\widetilde{M}_{\vec{r}}$.
Let $\vec{f} \in  L^{\vec{r}}_{\loc}(\R^d;\vec{X})$, fix a finite collection of dyadic cubes $\mc{D} \subseteq \ms{D}$ and set
\begin{equation*}
  A_0:= \sup_{\mc{D}\subseteq \ms{D} \text{ finite}}\,\nrmb{\widetilde{M}^{\mc{D}}_{\vec{r}}}_{L^{\vec{r}}(\R^d;\vec{X})  \to L^{r,\infty}(\R^d;X) },
\end{equation*}
which is finite by Lemma \ref{lemma:weakbiHL}.
For a cube $Q \in \mc{D}$, we define its stopping children $\ch_{\mc{D}}(Q)$ to be the collection of maximal cubes $Q' \in \mc{D}$ such that $Q' \subsetneq Q$ and
\begin{equation}\label{eq:stopping}
  \nrms{\sup_{\substack{P \in \mc{D} \\Q' \subseteq P\subseteq Q}}\prod_{j=1}^{m}\ip{f_j}_{r_j,P}}_{X}> 2^{1/r} A_0\,\prod_{j=1}^{m}\ipb{\nrm{f_j}_{X_j}}_{r_j,Q} .
\end{equation}
Let $\mc{S}^1$ be the maximal cubes in $\mc{D}$, define recursively $\mc{S}^{k+1}:= \bigcup_{Q \in \mc{S}^k} \ch_{\mc{D}}(Q)$ and set $\mc{S}: = \bigcup_{k=1}^\infty \mc{S}^k$.

Fix $Q \in \mc{S}$ and let $E_{Q}:=Q\setminus \bigcup_{Q'\in\ch_{\mc{D}}(Q)}Q'$. Define the set
  \begin{equation*}
    Q^* := \cbraces{x \in \R^d: \nrmb{\widetilde{M}_{\vec{r}}^{\mc{D}}\ha{ \vec{f}\ind_Q}(x) }_{X}  > 2^{1/r} A_0 \prod_{j=1}^{m}\ipb{\nrm{f_j}_{X_j}}_{r_j,Q}}.
  \end{equation*}
Then by the definition of $A_0$ we have
\begin{equation}\label{eq:temp1}
\abs{Q^*}^{1/r}\leq \frac{1}{2^{1/r}}\frac{\prod_{j=1}^{m}\nrm{f_j\ind_Q}_{L^{r_j}(\R^d;X_j)} }{\prod_{j=1}^{m}\ipb{\nrm{f_j}_{X_j}}_{r_j,Q}} =
\frac{1}{2^{1/r}} \abs{Q}^{1/r}.
\end{equation}
   Moreover, for $Q' \in \ch_{\mc{S}}(Q)$ and $x \in Q'$, we have  by \eqref{eq:stopping}
  \begin{align*}
    \nrmb{\widetilde{M}^{\mc{D}}_{\vec{r}}\ha{ \vec{f}\ind_Q}(x)}_{X}  &\geq \nrms{\sup_{\substack{P \in \mc{D}: \\ Q' \subseteq P\subseteq Q }}\prod_{j=1}^{m}\ip{f_j}_{r_j,P}}_{X}> 2^{1/r} A_0\,\prod_{j=1}^{m}\ipb{\nrm{f_j}_{X_j}}_{r_j,Q},
  \end{align*}
  so $x \in Q^*$ and thus $Q' \subseteq Q^*$. Using the disjointness of the cubes in $\ch_{\mc{D}}(Q)$ and \eqref{eq:temp1},  we get
  \begin{equation*}
    \sum_{Q' \in \ch_{\mc{S}}(Q)} |Q'| \leq |Q^*| \leq \frac{1}{2} |Q|.
  \end{equation*}
  So $|E_{Q}| \geq \frac{1}{2}|Q|$, which means that $\mc{S}$ is a sparse collection of dyadic cubes.

  Next, we check that $\widetilde{M}^{\mc{D}}_{\vec{r}}(\vec{f}\hspace{2pt})$ is pointwise dominated by the sparse operator associated to $\mc{S}$. For each $P \in \mc{D}$ we define
\begin{equation*}\parent_{\mc{S}}(P) \\
:= \cbrace{Q \in \mc{S}:Q \text{ minimal such that } P \subseteq Q},
\end{equation*}
which allows us to partition $\mc{D}$ as follows
\begin{equation*}
  \mc{D} = \bigcup_{Q \in \mc{S}} \cbraceb{P \in \mc{D}: \parent_\mc{S}(P)=Q}
\end{equation*}
    Fix $Q \in \mc{S}$, $x \in Q$ and let $Q' \in \mc{D}$ be the minimal cube such that $x \in Q'$ and $\parent_\mc{S}(Q')=Q$. If $Q' \subsetneq Q$ we have by the definition of $Q'$  that
\begin{align*}
  \nrms{\sup_{\substack{P \in \mc{D}:\\\parent_{\mc{S}}(P)=Q}}\prod_{j=1}^{m}\ip{f_j}_{r_j,P}\ind_{P}(x)}_{X}&=\nrms{\sup_{\substack{P \in \mc{D}: \\Q' \subseteq P\subseteq Q}} \prod_{j=1}^{m}\ip{f_j}_{r_j,P}}_{X}\\
  &\leq  2^{1/r} A_0 \prod_{j=1}^{m} \ipb{\nrm{f_j}_{X_j}}_{r_j,Q}  \ind_Q(x).
\end{align*}
If $Q'=Q$ the same estimate follows directly from the $r_j$-convexity of $X_j$ for $1\leq j\leq m$.
Using the fact that $\nrm{\cdot}_{\ell^\infty} \leq \nrm{\cdot}_{\ell^q}$ and the $q$-convexity of $X$ we can conclude for any $x \in \R^d$
  \begin{align*}
    \nrmb{\widetilde{M}^{\mc{D}}_{\vec{r}}(\vec{f}\hspace{2pt})(x)}_{X}&= \nrms{\sup_{Q \in \mc{S}} \sup_{\substack{P \in \mc{D}:\\\parent_{\mc{S}}(P)=Q}}\prod_{j=1}^{m}\ip{f_j}_{r_j,P}\ind_{P}(x)}_{X}\\
    &\leq  \has{\sum_{Q \in \mc{S}}\nrms{{ {\sup_{\substack{P \in \mc{D}:\\\parent_{\mc{S}}(P)=Q}}\prod_{j=1}^{m}\ip{f_j}_{r_j,P}\ind_{P}(x)}}}_{X}^q}^{1/q} \\
    &\leq  2^{1/r} A_0 \has{\sum_{Q \in \mc{S}}\prod_{j=1}^{m} \ipb{\nrm{f_j}_{X_j}}_{r_j,Q}^q  \ind_Q(x)}^{1/q}.
  \end{align*}
which is a pointwise sparse domination result for $\widetilde{M}^{\mc{D}}_{\vec{r}}$. Using the Fatou property of $X$ we know that
\begin{equation*}
  \nrmb{\widetilde{M}^{\ms{D}}_{\vec{r}}(\vec{f}\hspace{2pt})(x)}_X \leq \sup_{\mc{D}\subseteq \ms{D}:\mc{D} \text{ finite}}\nrmb{\widetilde{M}^{\mc{D}}_{\vec{r}}(\vec{f}\hspace{2pt})(x)}_X, \qquad x \in \R^d,
\end{equation*}
so our claim for $\widetilde{M}^{\ms{D}}_{\vec{r}}$ instead of $\widetilde{M}_{\vec{r}}$ follows from Proposition~\ref{prop:ptwisetoform}. Our claim for $\widetilde{M}_{\vec{r}}$ then follows from the three lattice lemma. The final statement follows by taking $g=\ind_{\R^d}$.
\end{proof}

\begin{remark}
Note that the proof of Theorem \ref{theorem:sparsedomHL} actually proves a pointwise sparse domination result for $\widetilde{M}_{\vec{r}}^{\mc{D}}$. Indeed, under the assumptions of Theorem \ref{theorem:sparsedomHL} we have that for any finite collection of dyadic cubes $\mc{D}\subseteq \ms{D}$ and $\vec{f} \in L^{\vec{r}}_{\loc}(\R^d;\vec{X})$ there exists a sparse collection of cubes $\mc{S} \subseteq \mc{D}$ such that
  \begin{equation*}
\nrmb{\widetilde{M}_{\vec{r}}^{\mc{D}}(\vec{f}\hspace{2pt})(x)}_{X} \lesssim_{X,r}  \has{\sum_{Q \in \mc{S}}\prod_{j=1}^{m}\ipb{\nrm{f_j}_{X_j}}_{r_j,Q}^q \ind_Q(x)}^{1/q} \qquad x\in \R^d.
\end{equation*}
\end{remark}

Using Proposition \ref{prop:wfromsdscalar}, Theorem \ref{theorem:sparsedomHL} and the density of simple functions we can now directly conclude weighted estimates for $\widetilde{M}_{\vec{r}}$. In particular this proves the $p$-independence of the $\vec{r}$-Hardy--Littlewood property.

\begin{corollary}\label{corollary:weightedHL}
Let $\vec{X}$ be an $m$-tuple of quasi-Banach function spaces, take $\vec{r} \in (0,\infty)^{m}$ and $q \in [r,\infty)$. Suppose $\vec{X} \in \HL_{\vec{r}}$ and assume $X$ is an order-continuous $q$-convex quasi-Banach function space. Then for $\vec{p} \in (0,\infty]^m$ with $\vec{r}<\vec{p}$ and $p<\infty$ and  any $w \in A_{\vec{p},(\vec{r},\infty)}$
\begin{equation*}
  \nrmb{\widetilde{M}_{\vec{r}}}_{L^{\vec{p}}_{\vec{w}}(\R^d;\vec{X})\to L_w^p(\R^d;X)}\lesssim_{\vec{X},\vec{p},q,\vec{r}} [\vec{w}]_{\vec{p},(\vec{r},\infty)}^{ \max\cbraceb{\frac{\frac{1}{\vec{r}}}{\frac{1}{\vec{r}}- \frac{1}{\vec{p}}},\frac{p}{q}}}.
\end{equation*}
\end{corollary}
We point out that the condition $p<\infty$ here is necessary. Indeed, it is shown in \cite[Remark 2.9]{GMT93} that $\widetilde{M}$ is not bounded on $L^\infty(\R;\ell^2)$.

\begin{remark}\label{rem:sht}
  The arguments presented in this section also go through in a space of homogeneous type instead of $\R^d$ with the Lebesgue measure, provided one uses the dyadic cubes from e.g. \cite{HK12}. Furthermore, for the dyadic counterpart of $\widetilde{M}_{\vec{r}}$ one can also work on $\R^d$ with any locally finite measure $\mu$, see \cite{HL17} for details.
\end{remark}

\section{A limited range multilinear \texorpdfstring{$\UMD$}{UMD} property for quasi-Banach function spaces}\label{section:multiUMD}
 A Banach space has the $\UMD$ property if the martingale difference sequence of any finite martingale in $L^p(\Omega;X)$ is unconditional for some (equivalently all) $p \in (1,\infty)$, i.e. if for $(f_k)_{k=0}^n$ any finite martingale in $L^p(\Omega;X)$ for some (equivalently all) $p \in (1,\infty)$ and a probability space $(\Omega,\P)$ and all scalars $\abs{\epsilon_0}=\abs{\epsilon_n}=1$ we have
 \begin{equation}\label{eq:UMD}
   \nrms{\sum_{k=1}^n \epsilon_k df_n}_{L^p(\Omega;X)} \lesssim \nrms{\sum_{k=1}^n df_n}_{L^p(\Omega;X)},
 \end{equation}
 where $(df_k)_{k=1}^n$ is the difference sequence of $(f_k)_{k=0}^n$. The least admissible constant in \eqref{eq:UMD} is denoted by $\beta_{p,X}$.
The class of $\UMD$ Banach function spaces includes for example all reflexive Lebesgue, Lorentz and Musielak--Orlicz spaces. As the $\UMD$ property implies reflexivity, $L^1$ and $L^\infty$ do not have the $\UMD$ property. For an introduction to the $\UMD$ property we refer the reader to \cite{HNVW16,Pi16}.

As already noted in the previous section, for Banach function spaces the $\UMD$ property is intimately connected to the Hardy--Littlewood property. As shown by Bourgain \cite{Bo84} and Rubio de Francia \cite[Theorem 3]{Ru86}, a Banach function space $X$ has the $\UMD$ property if and only if both $X$ and $X^*$ have the Hardy--Littlewood property. This connection between the Hardy--Littlewood property and the $\UMD$ property is made quantitative in \cite{KLW20}, where it is shown that $\nrm{\widetilde{M}}_{p,X} \lesssim (\beta_{p,X})^2$.

Motivated by this connection between the Hardy--Littlewood property and the $\UMD$ property and using the extension of the Hardy--Littlewood property to the rescaled, multilinear setting from Section \ref{section:multiHL}, we will now define a limited range, multilinear version of the $\UMD$ property for $m$-tuples of quasi-Banach function spaces.

\begin{definition}
  Let $\vec{X}$ be an $m$-tuple of quasi-Banach function spaces, take $\vec{r} \in (0,\infty)^m$ and $s \in (r,\infty]$. We say that $\vec{X}$ has the $(\vec{r},s)$-$\UMD$ property and write $\vec{X} \in \UMD_{\vec{r},s}$ if $X = \prod_{j=1}^m X_j$ is an order-continuous Banach function space and $(\vec{X},X^*) \in \HL_{(\vec{r},s')}$.
\end{definition}

Note that while the $\UMD$ property is well-defined in terms of martingale difference sequences for any Banach space, our limited range multilinear version is only given for quasi-Banach \textit{function} spaces and has no immediate connection to martingales. It would be interesting to have an equivalent characterization of either the limited range or the multilinear generalization (for example in terms of martingale difference sequences) that does not use the lattice structure of $\vec{X}$.

\begin{remark}
 In the linear case the class $\UMD_{2,\infty}$ has already been implicitly used in \cite{PSX12} in connection to the so-called \emph{Littlewood--Paley--Rubio de Francia property} for Banach function spaces. Indeed, although \cite[Theorem 3.1]{PSX12} assumes $X^2$ to have the $\UMD$ property, the proof only uses $X^2,X^* \in \HL$, which by Proposition \ref{prop:hlcharofumd} is equivalent to $X\in\UMD_{2,\infty}$.
\end{remark}

 \bigskip

 As a first result on the limited range multilinear $\UMD$ property we will show that
our nomenclature makes sense, i.e. that the $\UMD_{\vec{r},s}$ property is actually related to the $\UMD$ property for Banach function spaces. If $X$ is a Banach function space, then $X$ has the $\UMD$ property if and only if $X \in \UMD_{1,\infty}$. This follows directly from the result of Bourgain and Rubio de Francia and the case $m=r=s'=1$, of the following proposition.

\begin{proposition}\label{prop:hlcharofumd}
Let $\vec{X}$ be an $m$-tuple of quasi-Banach function spaces and let $\vec{r}\in[1,\infty)^m$ and $s\in(1,\infty]$. The following are equivalent:
\begin{enumerate}[(i)]
\item\label{it:hlchar1}$\vec{X}\in\UMD_{\vec{r},s}$;
\item\label{it:hlchar2}  $\vec{X}\in\HL_{\vec{r}}$ and $(X_1,\ldots, X_{j-1},X_{j+1},\ldots,X_m,X^\ast)\in\HL_{(r_1,\ldots,r_{j-1},r_{j+1},\ldots r_m,s')}$ for all $j\in\{1,\ldots,m\}$.
\end{enumerate}
\end{proposition}
\begin{proof}
For \ref{it:hlchar1}$\Rightarrow$\ref{it:hlchar2} we only prove $\vec{X}\in\HL_{\vec{r}}$. The other results with $j\in\{1,\ldots,m\}$ follow from an analogous argument by interchanging the roles of $X^\ast$ and $X_j$ and the roles of $s'$ and $r_j$.

Let $(\Omega,\mu)$ denote the underlying  measure space over which the $\vec{X}$ are defined and fix $\vec{p}\in(0,\infty]^m$ with $\vec{r}<\vec{p}$, $1\leq p<s$ and a finite collection of cubes $\mc{D}$. By the pointwise sparse domination result for $M^{\mc{D}}_{\vec{r}}$, it follows from Proposition~\ref{prop:ptwisetoform} that $$\|M^\mc{D}_{\vec{r}}(\vec{f}\hspace{2pt})g\|_{L^1(\R^d)}\lesssim_{\vec{r}}\|M^\mc{D}_{(\vec{r},1)} (\vec{f},g)\|_{L^1(\R^d)}$$ for $\vec{f}\in L^{\vec{r}}_{\loc}(\R^d)$, $g\in L^1_{\loc}(\R^d)$. Since $\widetilde{M}^\mc{D}_{\vec{r}}(\vec{f})(x,\omega) =M^\mc{D}_{\vec{r}}(\vec{f}(\cdot,\omega))(x)$, combining this with Fubini's theorem  we obtain for $\vec{f} \in L^{\vec{p}}_c(\R^d;\vec{X})$ and $g \in L^{p'}(\R^d;X^*)$
\begin{align*}
\left|\int_{\R^d}\int_\Omega\!\widetilde{M}^\mc{D}_{\vec{r}}(\vec{f}\hspace{2pt})g\,\mathrm{d}\mu\,\mathrm{d}x\right|&\leq\int_\Omega\!\|M^\mc{D}_{\vec{r}}(\vec{f}(\cdot,\omega))g(\cdot,\omega)\|_{L^1(\R^d)}\,\mathrm{d}\mu(\omega)\\
&\lesssim_{\vec{r}}\int_\Omega\nrmb{M^\mc{D}_{(\vec{r},1)}(\vec{f}(\cdot, \omega),g(\cdot,\omega))}_{L^1(\R^d)}\,\mathrm{d}\mu(\omega)\\
&=\nrmb{\widetilde{M}^\mc{D}_{(\vec{r},1)}(\vec{f}\hspace{2pt},g)}_{L^1(\R^d;L^1(\Omega))} \leq\nrmb{\widetilde{M}^\mc{D}_{(\vec{r},s')}(\vec{f}\hspace{2pt},g)}_{L^1(\R^d;L^1(\Omega))}\\
&\leq\nrmb{\widetilde{M}_{(\vec{r},s')}}_{(\vec{p},p'),(\vec{X},X^\ast)}\Big(\prod_{j=1}^m\|f_j\|_{L^{p_j}(\R^d;X_j)}\Big)\|g\|_{L^{p'}(\R^d;X^\ast)},
\end{align*}
where in the second to last step we used H\"older's inequality with $s'\geq 1$ and \ref{it:hlchar1} and Corollary~\ref{corollary:weightedHL} in the last. Taking a supremum over all $g \in L^{p'}(\R^d;X^*)$ with $\|g\|_{L^{p'}(\R^d;X^\ast)}=1$ proves that $\vec{X}\in\HL_{\vec{r}}$, as asserted.

The proof of \ref{it:hlchar2}$\Rightarrow$\ref{it:hlchar1} relies on some combinatorics. To facilitate this, we set $r_{m+1}:=s'$ and $X_{m+1}:=X^\ast$. Fix $\vec{p}\in(0,\infty]^{m+1}$ with $\min\vec{p}>\max\vec{r}$, $\vec{f}\in L^{\vec{p}}(\R^d;\vec{X})$, and a finite collection of cubes $\mc{D}$. Note that
\[
\prod_{j=1}^{m+1}\langle f_j\rangle_{r_j,Q}\ind_Q=\prod_{j=1}^{m+1}\Big(\prod_{\substack{k=1\\k\neq j}}^{m+1}\langle f_k\rangle_{r_k,Q}\ind_Q\Big)^{\frac{1}{m}}
\]
for all $Q\in\mc{D}$ so that
\[
\widetilde{M}_{\vec{r}}^\mc{D}(\vec{f}\hspace{2pt})\leq\prod_{j=1}^{m+1}\widetilde{M}^\mc{D}_{\vec{q}}(\vec{g}\hspace{1pt})^{\frac{1}{m}}
\]
with
\begin{align*}
  \vec{q}&=(r_1,\ldots,r_{j-1},r_{j+1},\ldots r_{m+1})\\ \vec{g} &= (f_1,\ldots,f_{j-1},f_{j+1},\ldots,f_{m+1})
\end{align*}
Furthermore setting $\vec{Y_j}=(X_1,\cdots,X_{j-1},X_{j+1},\cdots, X_{m+1})$, we have
\[
\prod_{j=1}^{m+1}Y_j^{\frac{1}{m}}=\prod_{j=1}^{m+1}\prod_{\substack{k=1\\k\neq j}}^{m+1}X_k^{\frac{1}{m}}=\prod_{j=1}^{m+1}X_j=L^1(\Omega).
\]
Thus setting
$A_j:=\|\widetilde{M}_{\vec{q}}\|_{(p_j,\ldots,p_j),\vec{Y_j}}$, which is finite by Corollary \ref{corollary:weightedHL}, we have
\begin{align*}
\nrmb{M_{\vec{r}}^\mc{D}(\vec{f}\hspace{2pt})}_{L^p(\R^d;L^1(\Omega))}&\leq\prod_{j=1}^{m+1}\nrmb{ \widetilde{M}^\mc{D}_{\vec{q}}(\vec{g})^{\frac{1}{m}} }_{L^{p_j}(\R^d;Y_j^{\frac{1}{m}})}=\prod_{j=1}^{m+1}\nrmb{\widetilde{M}^\mc{D}_{\vec{q}}(\vec{g})}^{\frac{1}{m} }_{L^{\frac{p_j}{m}}(\R^d;Y_j)}\\
&\leq\prod_{j=1}^{m+1}A_j^{\frac{1}{m}} \prod_{\substack{k=1\\k\neq j}}^{m+1}\|f_k\|_{L^{p_j}(\R^d;X_k)}^{\frac{1}{m}}=\prod_{j=1}^{m+1}A_j^{\frac{1}{m}}\|f_k\|_{L^{p_j}(\R^d;X_k)},
\end{align*}
proving \ref{it:hlchar1}. The assertion follows.
\end{proof}

\begin{example}
Let $(\Omega,\mu)$ be a $\sigma$-finite measure space. In the case $m=1$, it follows from Proposition~\ref{prop:hlcharofumd} that $X\in\UMD_{r,s}$ for $1\leq r<s\leq\infty$ if and only if $X^r\in\HL$ and $(X^\ast)^{s'}\in\HL$. This implies the following:
\begin{enumerate}[(i)]
\item If $X=L^p(\Omega)$ with $p\in(r,s)$, then $X\in\UMD_{r,s}$.
\item If $X=L^{p,q}(\Omega)$ with $p,q\in(r,s)$, then $X\in\UMD_{r,s}$.
\item If $X=L^\Phi(\Omega)$ is a Musielak-Orlicz space such that $(\omega,t)\mapsto\Phi(\omega,t^{\frac{1}{r}})$ and $(\omega,t)\mapsto\Phi^\ast(\omega,t^{\frac{1}{s'}})$ are Young functions satisfying the $\Delta_2$ condition, then $X\in\UMD_{r,s}$. See \cite{FG91, LVY18} for the $\UMD$ (and thus the $\HL$) property of these spaces.
\end{enumerate}
\end{example}

In \cite{LN19} vector-valued extensions of multilinear operators in quasi-Banach function spaces were constructed through weighted techniques. In that work the
condition that $((X_j^{r_j})^*)^{(s_j/r_j)'}$  has the $\UMD$ property for $1\leq j\leq m$ was imposed. In the next proposition we wish to compare this assumption to our limited range multilinear $\UMD$ property.

\begin{proposition}\label{proposition:comparisonoldassumption}
  Let $\vec{X}$ be an $m$-tuple of quasi-Banach function spaces, let $\vec{r}\in(0,\infty)^m$ and take $\vec{r}<\vec{s}\leq \infty$. Suppose that $X_j$ is $r_j$-convex, $s_j$-concave and
  $\hab{(X_j^{r_j})^*}^{(s_j/r_j)'}$
  has the $\UMD$ property for $1\leq j \leq m$. Then for all $q\in(0,r]$ we have $\vec{X}^q\in\UMD_{\frac{\vec{r}}{q},\frac{s}{q}}$. In particular, $\vec{X} \in \UMD_{\vec{r},s}$ if $r\geq 1$.
\end{proposition}

\begin{proof}
Note that $\vec{X}^q\in\UMD_{\frac{\vec{r}}{q},\frac{s}{q}}$ per definition means that $$(X_1^q,\ldots,X_m^q,(X^q)^\ast)\in\HL_{\big(\frac{\vec{r}}{q},\big(\frac{s}{q}\big)'\big)}.$$ So by Proposition \ref{prop:HLlintomulti} it suffices to show $(X_j^q)^{\frac{r_j}{q}}=X_j^{r_j} \in \HL$ for $j=1,\cdots,m$ and $((X^q)^*)^{(s/q)'} \in \HL$.
Since $(s_j/r_j)'\geq  1$, we know that $(X_j^{r_j})^*$ has the $\UMD$ property (see \cite[Theorem III.4]{Ru86}) and thus $X_j^{r_j} \in \HL$ for $j=1,\cdots,m$.
To show $((X^q)^*)^{(s/q)'} \in \HL$ we note that by \cite[Proposition 3.4]{LN19} we have $((X^r)^*)^{(s/r)'} \in \UMD$. Then, by \cite[Proposition 3.3(iii)]{LN19} this implies that also $((X^q)^*)^{(s/q)'} \in \UMD$ for all $q\in(0,r]$. In particular, we have $((X^q)^*)^{(s/q)'} \in \HL$, as desired. The assertion follows.
\end{proof}

the $\UMD_{\vec{r},s}$ class is stable under iteration. Recall that for quasi-Banach function spaces $X$ and $Y$ respectively over measure spaces $(\Omega_1,\mu_1)$ and $(\Omega_2,\mu_2)$ the mixed-norm space $X(Y)$ is given by all measurable functions $f\colon \Omega_1\times \Omega_2 \to \C$ such that
\begin{equation*}
  \nrmb{\omega_1 \mapsto \nrm{f(\omega_1,\cdot)}_Y}_X<\infty.
\end{equation*}

\begin{proposition}\label{prop:iterateumd}
Let $\vec{r} \in (0,\infty)^m$ and $s \in (1,\infty]$ and let $\vec{X}$ and $\vec{Y}$ be $m$-tuples of quasi-Banach function spaces. If $\vec{X},\vec{Y}\in \UMD_{\vec{r},s}$, then $\vec{X}(\vec{Y})\in\UMD_{\vec{r},s}$.
\end{proposition}
\begin{proof}
Denote by $(\Omega_1,\mu_1)$, $(\Omega_2,\mu_2)$ the $\sigma$-finite measure spaces that $\vec{X}$, $\vec{Y}$ are respectively defined over and write
\[
  A_1:= \sup_{\mc{D}\subseteq \ms{D} \text{ finite}}\,\nrmb{\widetilde{M}^{\mc{D}}_{\vec{r}}}_{L^{\vec{r}}(\R^d;\vec{X})  \to L^{r,\infty}(\R^d;X) },\quad  A_2:= \sup_{\mc{D}\subseteq \ms{D} \text{ finite}}\,\nrmb{\widetilde{M}^{\mc{D}}_{\vec{r}}}_{L^{\vec{r}}(\R^d;\vec{Y})  \to L^{r,\infty}(\R^d;Y) }.
\]
Let $\mc{D}$ denote a finite collection of cubes and let $\vec{f}\in L^\infty_c(\R^d;\vec{X}(\vec{Y}))$. By Fubini's Theorem and by applying Theorem~\ref{theorem:sparsedomHL} twice we obtain
\begin{align*}
\nrmb{\widetilde{M}^\mc{D}_{\vec{r},s}(\vec{f},g)}_{L^1(\R^d;L^1(\Omega_1\times\Omega_2))}&=\int_{\Omega_1} \nrmb{\widetilde{M}^\mc{D}_{\vec{r},s}(\vec{f}(\cdot,\omega_1,\cdot),g(\cdot,\omega_1,\cdot)) }_{L^1(\R^d;L^1(\Omega_1))}\,\mathrm{d}\mu_1(\omega_1)\\
&\lesssim A_2\int_{\Omega_1}\nrmb{\widetilde{M}^\mc{D}_{\vec{r},s}(\|\vec{f}(\cdot,\omega_1,\cdot)\|_{\vec{Y}}, \|g(\cdot,\omega_1,\cdot)\|_{Y^\ast})}_{L^1(\R^d)}\,\mathrm{d}\mu_1(\omega_1)\\
&=A_1\nrmb{\widetilde{M}^\mc{D}_{\vec{r},s}(\|\vec{f}\|_{\vec{Y}},\|g\|_{Y^\ast})}_{L^1(\R^d;L^1(\Omega_1))}\\
&\lesssim A_1A_2\nrmb{M^\mc{D}_{\vec{r},s}(\|\vec{f}\|_{\vec{X}(\vec{Y})},\|g\|_{X^\ast(Y^\ast)})}_{L^1(\R^d)}.
\end{align*}
Thus, by Proposition~\ref{prop:multimaxwest} and a density argument we conclude that $\vec{X}(\vec{Y})\in\UMD_{\vec{r},s}$, as desired.
\end{proof}

Next we show that we can add $L^\infty$ spaces to existing $\UMD$ tuples to create new ones. Note in particular that in the case $m=2$, this following result implies that if $X$ has the $\UMD$ property, then $(X,L^\infty(\Omega))\in\UMD_{(1,1),\infty}$.
\begin{proposition}
Let $\vec{r} \in (0,\infty)^m$ and $s \in (1,\infty]$. Let $\vec{X}$ be an $m-1$-tuple of quasi-Banach function spaces over a measure space $\Omega$. If $\vec{X}\in\UMD_{(r_1,\ldots,r_{m-1}),s}$, then
\[
(X_1,\ldots,X_{m-1},L^\infty(\Omega))\in\UMD_{\vec{r},s}.
\]
\end{proposition}
\begin{proof}
We first note that by Proposition~\ref{example:products}\ref{it:product0} we have $\big(\prod_{j=1}^{m-1}X_j\big)\cdot L^\infty(\Omega)=X\cdot L^\infty(\Omega)=X$.
Next, let $\mc{D}$ denote a finite collection of cubes and fix $\vec{p}\in(1,\infty]^m$ with $p_m=\infty$ and $\vec{p}>\vec{r}$, $p<s$. For $\vec{f}\in L^{\vec{p}}(\R^d;\vec{X})$, $g\in L^{p'}(\R^d;X^\ast)$ we have $$\widetilde{M}^{\mc{D}}_{(\vec{r},s')}(\vec{f},g)\leq \widetilde{M}^{\mc{D}}_{(r_1,\ldots,r_{m-1},s')}(f_1,\ldots,f_{m-1},g) \widetilde{M}^{\mc{D}}_{r_m}(f_m).$$ Hence,
\begin{align*}
&\|\widetilde{M}^{\mc{D}}_{(\vec{r},s')}(\vec{f},g)\|_{L^1(\R^d;L^1(\Omega))}\\
&\leq\|\widetilde{M}^{\mc{D}}_{(r_1,\ldots,r_{m-1},s')}(f_1,\ldots,f_{m-1},g) \|_{L^1(\R^d;L^1(\Omega))}\|\widetilde{M}^{\mc{D}} _{r_m}(f_m)\|_{L^{\infty}(\R^d;L^\infty(\Omega))}\\
&\leq\|\widetilde{M}_{(r_1,\ldots,r_{m-1},s')}\|_{(p_1,\ldots,p_{m-1},p'),\vec{X}} \Big(\prod_{j=1}^{m-1}\|f_j\|_{L^{p_j}(\R^d;X_j)}\Big)\|f_m\|_{L^\infty(\R^d;L^\infty(\Omega))}\|g\|_{L^{p'}(\R^d;X^\ast)},
\end{align*}
proving that $(X_1,\ldots,X_{m-1},L^\infty(\Omega))\in\UMD_{\vec{r},s}$. The assertion follows.
\end{proof}

To end this section we will give a family of examples in the form of iterated $L^p$-spaces that belong to the $\UMD_{\vec{r},s}$-class.
\begin{proposition}\label{prop:lebesgueumd}
Let $\vec{r} \in (0,\infty)^m$ and $s \in (1,\infty]$. Let $K\in\N$ and let $\vec{t}\hspace{2pt}^1,\ldots,\vec{t}\hspace{2pt}^K\in(0,\infty]^m$ with $\vec{t}\hspace{2pt}^k>\vec{r}$ and $1\leq t^k<s$ for all $k\in\{1,\cdots,K\}$. Let $(\Omega_k,\mu_k)$ for $k\in\{1,\cdots,K\}$ be $\sigma$-finite measure spaces and for $j\in\{1,\cdots,m\}$ we set $$X_j:=L^{t^1_j}(\Omega_1;\cdots;L^{t^K_j}(\Omega_K)).$$
 Then $\vec{X}\in\UMD_{\vec{r},s}$.
\end{proposition}

\begin{proof}
By Proposition \ref{prop:iterateumd} it suffices to consider the case $K=1$ and write $\vec{t}^1 = \vec{t}$ and $\vec{r}^1=\vec{r}$.
Write $X_j=L^{t_j}(\Omega)$ so that $X=\prod_{j=1}^mX_j=L^t(\Omega)$ by Proposition~\ref{example:products}\ref{it:product1}. Note that since $\frac{t_1}{r_1},\ldots,\frac{t_m}{r_m},\frac{t'}{s'}\in(1,\infty]$, we have $X_j^{r_j}=L^{\frac{t_j}{r_j}}(\Omega)\in\HL$ for all $j\in\{1,\ldots,m\}$ and $(X^\ast)^{s'}=L^{\frac{t'}{s'}}(\Omega)\in\HL$. Thus, it follows from Proposition~\ref{prop:HLlintomulti} that $\vec{X}\in\UMD_{\vec{r},s}$. The assertion follows.
\end{proof}

The main interest in the above result is that we can go beyond assuming that each individual $X_j$ has the $\UMD$ property. We can even consider examples such as $\ell^\infty(\ell^2)$, which by \cite[Proposition 8.l]{NVW15} does not even satisfy the Hardy-Littlewood property.

\begin{remark}
By mimicking the proof of Proposition~\ref{prop:lebesgueumd} we can also obtain a version of Proposition~\ref{prop:lebesgueumd} for Lorentz and Orlicz spaces. We point out however that it is not clear if we can consider the appropriate endpoint cases outside of the range of $\UMD$ spaces. More precisely, in the case of Lorentz spaces it is unknown whether $L^{p,\infty}(\Omega)$ for $p\in(1,\infty)$ satisfies the Hardy-Littlewood property.
 Similarly it is unknown whether there are Orlicz spaces  that are not $\UMD$, but satisfy the Hardy-Littlewood property. If there are such spaces, we obtain more examples beyond the setting of individual $\UMD$ conditions that fall within our range.
\end{remark}

\section{Main results}\label{sec:proofs}
In this section we state and prove our main results. We will first show that scalar-valued sparse domination implies vector-valued sparse domination.
In view of Proposition~\ref{prop:sparsequiv}, Theorem~\ref{theorem:mainmultilinear} is a consequence of the following result. Note that we introduce the parameter $q$ into the theorem here, which is essential in obtaining the full range of vector-valued bounds, including the quasi-Banach range. We elaborate further on this in Section~\ref{sec:applymain}. Recall Convention \ref{conv:prod} for the definition of $p,r$ and $X$.
\begin{theorem}\label{thm:mainthm1}
Let $\vec{r}\in(0,\infty)^m$, $q\in(0,\infty)$, $s\in(q,\infty]$ and let $T$ be an operator defined on $m$-tuples of functions such that for any $\vec{f},g\in L_c^\infty(\R^d)$
\begin{equation}\label{eq:mainthmin}
\nrmb{T(\vec{f}\hspace{2pt})\cdot g}_{L^q(\R^d)}\leq C_T\nrmb{M_{(\vec{r},\frac{1}{\frac{1}{q}-\frac{1}{s}})}(\vec{f},g)}_{L^q(\R^d)}.
\end{equation}
Let $\vec{X}$ be and $m$-tuple of quasi-Banach function spaces over a measure space $(\Omega,\mu)$ such that $\vec{X}^q\in\UMD_{\frac{\vec{r}}{q},\frac{s}{q}}$. Furthermore suppose that for all simple functions $\vec{f}\in L^\infty_c(\R^d;\vec{X})$ the function $\widetilde{T}(\vec{f}\hspace{2pt}):\R^d \to X$ given by
  \begin{equation*}
    \widetilde{T}(\vec{f}\hspace{2pt})(x,\omega):= T(\vec{f}(\cdot,\omega))(x), \qquad (x,\omega) \in \R^d \times \Omega
  \end{equation*}
  is well-defined and strongly measurable. Then for all simple functions $\vec{f}\in L^\infty_c(\R^d;\vec{X})$ and  $g\in L^\infty_c(\R^d)$
\begin{equation}\label{eq:mainthmout}
\nrmb{\|\widetilde{T}(\vec{f}\hspace{2pt})\|_X\cdot g}_{L^q(\R^d)}\lesssim_{\vec{X},q,\vec{r},s} C_T\nrmb{M_{(\vec{r},\frac{1}{\frac{1}{q}-\frac{1}{s}})}(\|\vec{f}\|_{\vec{X}},g)}_{L^q(\R^d)}.
\end{equation}
\end{theorem}
Note that if $T$ is $m$-linear, then $\widetilde{T}(\vec{f}\hspace{2pt})$  is always well-defined and strongly measurable for simple functions $\vec{f}\in L^\infty_c(\R^d;\vec{X})$ through the tensor extension.

\begin{proof}
The proof essentially consists of applying Fubini's Theorem twice and then using the vector-valued sparse domination result for the multisublinear maximal operator. Let $\vec{f}\in L^\infty_c(\R^d;\vec{X})$ and  $g\in L^\infty_c(\R^d;((X^q)^\ast)^{\frac{1}{q}})$ be simple. Then for a.e. $\omega\in\Omega$ we have $f_j(\cdot,\omega),g(\cdot,\omega)\in L^\infty_c(\R^d)$. Thus, using Fubini's Theorem and \eqref{eq:mainthmin}, we have
\begin{equation}\label{eq:fubinitwice}
\begin{split}
\nrmb{\widetilde{T}(\vec{f}\hspace{2pt})\cdot g}_{L^q(\R^d;L^q(\Omega))}&=\nrmb{\omega\mapsto\|T(\vec{f} (\cdot,\omega),g(\cdot,\omega)\|_{L^q(\R^d)}}_{L^q(\Omega)}\\
&\leq C_T\nrmb{\omega\mapsto\|M_{(\vec{r},\frac{1}{\frac{1}{q}-\frac{1}{s}})}(\vec{f}(\cdot,\omega), g(\cdot,\omega))\|_{L^q(\R^d)}}_{L^q(\Omega)}\\
&=C_T\nrmb{\widetilde{M}_{(\vec{r},\frac{1}{\frac{1}{q}-\frac{1}{s}})}(\vec{f},g)}_{L^q(\R^d;L^q(\Omega))}.
\end{split}
\end{equation}
We set $X_{m+1}:=((X^q)^\ast)^{\frac{1}{q}}$ so that
\[
\prod_{j=1}^{m+1}X_j=(X^q\cdot(X^q)^\ast)^{\frac{1}{q}}=L^1(\Omega)^{\frac{1}{q}}=L^q(\Omega),
\]
which is an order-continuous $q$-convex quasi-Banach function space. Then it follows from the sparse domination result in Theorem~\ref{theorem:sparsedomHL} that
\[
\nrmb{\widetilde{M}_{(\vec{r},\frac{1}{\frac{1}{q}-\frac{1}{s}})}(\vec{f},g)}_{L^q(\R^d;L^q(\Omega))} \lesssim_{\vec{X},q,\vec{r},s}\nrmb{M_{(\vec{r},\frac{1}{\frac{1}{q}-\frac{1}{s}})}\hab{ \nrm{\vec{f}}_{\vec{X}},\|g\|_{((X^q)^\ast)^{\frac{1}{q}}}}}_{L^q(\R^d)}.
\]
By combining this with \eqref{eq:fubinitwice} and Proposition~\ref{prop:vvsdequivalence}, the assertion follows.
\end{proof}

We will now use Theorem \ref{thm:mainthm1} to deduce weighted boundedness for the vector-valued extension of an operator $T$ from a scalar-valued sparse domination result for $T$, which is new even in the unweighted setting. In view of Proposition~\ref{prop:sparsequiv}, Corollary~\ref{corollary:weights1} is a consequence of the following result with $q=1$.
\begin{theorem}\label{thm:weightcor}
Let $\vec{r}\in(0,\infty)^m$, $q\in(0,\infty)$, $s\in(q,\infty]$ and let $T$ an $m$-linear or positive-valued $m$-sublinear operator satisfying \eqref{eq:mainthmin} and let $\vec{X}$ satisfy the assumptions in Theorem~\ref{thm:mainthm1}. Then for all $p \in (0,\infty]^m$ with $\vec{p}>\vec{r}$ and $p<s$, and all $\vec{w}\in A_{\vec{p},(\vec{r},s)}$ we have
\[
\nrmb{\widetilde{T}(\vec{f}\hspace{2pt})}_{L_w^p(\R^d;X)}\lesssim_{\vec{X},\vec{p},q,\vec{r},s} C_T[\vec{w}]_{\vec{p},(\vec{r},s)}^{ \max\left\{\frac{\frac{1}{\vec{r}}}{\frac{1}{\vec{r}}-\frac{1}{\vec{p}}},\frac{\frac{1}{q}-\frac{1}{s}}{\frac{1}{p}-\frac{1}{s}}\right\}}\prod_{j=1}^m\|f_j\|_{L^{p_j}_{w_j}(\R^d;X_j)}
\]
for all $\vec{f}\in L^{\vec{p}}_{\vec{w}}(\R^d;\vec{X})$.
\end{theorem}

\begin{proof}
This follows from combining Theorem~\ref{thm:mainthm1} with Proposition~\ref{prop:wfromsdscalar}.
\end{proof}
\begin{remark}
As a consequence of the fact that the sparse domination of the lattice multisublinear maximal operator also holds in spaces of homogeneous type (see Remark~\ref{rem:sht}), the results in this section also hold in spaces of homogeneous type.
\end{remark}
\section{Applications}\label{sec:applymain}
In this section we provide a discussion regarding obtaining vector-valued estimates beyond the Banach range. Furthermore, we provide applications of our result specifically to multilinear Calder\'on-Zygmund operator and the bilinear Hilbert transform. We point out that our results are applicable far beyond these examples as they include all operators satisfying sparse domination, but we restrict ourselves to these examples as they highlight the utility of our techniques in various settings.
\subsection{Vector-valued estimates in the quasi-Banach range}
In the multilinear setting it is a natural occurrence that an operator maps into a Lebesgue space with exponents smaller than $1$ and hence, no longer in the Banach range. For this reason one also expects the vector-valued extensions of the operator to map into spaces in the quasi-Banach range. However, in our multilinear $\UMD$ condition we assume that the product of the spaces in a tuple is a Banach space. This is partly because we are obtaining our estimates after a dualization argument which is usually possible in the quasi-Banach setting. Thanks to the quantitative extrapolation result \cite{Ni19} this does not hinder us in obtaining sharp bounds in the full range of exponents, but we are still hindered in how much convexity we are allowed to assume on our spaces.

In this subsection we explain how the parameter $q$ in our main theorems can be used to recover the expected results in the quasi-Banach range, at the cost of a worse exponent in the weighted estimate. We illustrate this in the following proposition:
\begin{proposition}\label{prop:weightsellt}
Let $\vec{r}\in(0,\infty)^m$, $q_0\in(0,\infty)$, and let $T$ be an $m$-linear operator. Suppose that for any $f_1,\cdots,f_m \in L^\infty_c(\R^d)$ there exists a sparse collection $\mc{S}$ such that
\begin{equation}\label{eq:pointwisesparseq}
   \abs{T(\vec{f}\hspace{2pt})}\leq C_T\, \has{\sum_{Q\in \mc{S}}\has{ \prod_{j=1}^m\ip{f_j}_{r_j,Q}^{q_0}} \ind_Q}^{1/q_0}.
\end{equation}
Then for all $\vec{p}\in(0,\infty]^m$, $\vec{t}\in(0,\infty]^m$ with $\vec{r}< \vec{p},\vec{t}$ and $p,t<\infty$ and all $\vec{w}\in A_{\vec{p},(\vec{r},\infty)}$ the tensor extension $\widetilde{T}$ of $T$ satisfies
\begin{equation}\label{eq:pointwisesparsecons}
\nrmb{\widetilde{T}}_{L_{w_1}^{p_1}(\R^d;\ell^{t_1})\times\cdots\times L_{w_m}^{p_m}(\R^d;\ell^{t_m})\to L_w^p(\R^d;\ell^t)}\lesssim_{\vec{p},q_0,\vec{r},\vec{t}} C_T[\vec{w}]_{\vec{p},(\vec{r},\infty)}^\gamma,
\end{equation}
with
\[
\gamma=\begin{cases} \max\left\{\frac{\frac{1}{\vec{r}}}{\frac{1}{\vec{r}}-\frac{1}{\vec{p}}},\frac{\frac{1}{q_0}}{\frac{1}{p}}\right\} &\text{if $t\in[q_0,\infty)$;}\\[0.5cm]
\max\left\{\frac{\frac{1}{\vec{r}}}{\frac{1}{\vec{r}}-\frac{1}{\vec{p}}},\frac{\frac{1}{t}}{\frac{1}{p}}\right\}  &\text{if $t\in(r,q_0]$.}
\end{cases}
\]
\end{proposition}
The above result also holds for $m$-sublinear $T$ as in Theorem~\ref{thm:mainthm1}. Of course, our methods go much beyond the setting of $\ell^t$-spaces, but we restrict our attention to this particular case for now for the sake of clarity and for the sake of comparing our result to previous results in the literature.
\begin{proof}
Write $\vec{X}=(\ell^{t_1},\ldots,\ell^{t_m})$. We consider the two cases separately.

For the case $t\in[q_0,\infty)$, we note that $\vec{X}^{q_0}=(\ell^{\frac{t_1}{q_0}},\ldots,\ell^{\frac{t_m}{q_0}})\in\UMD_{\frac{\vec{r}}{q_0},\infty}$ by Proposition~\ref{prop:lebesgueumd}. Thus, the result follows from an application of Proposition~\ref{prop:ptwisetoform} and  Theorem~\ref{thm:weightcor} with $q=q_0$ and $s=\infty$.

In the other case $t\in(r,q_0]$ we have $\vec{X}^t=(\ell^{\frac{t_1}{t}},\ldots,\ell^{\frac{t_m}{t}})\in\UMD_{\frac{\vec{r}}{t},\infty}$ by Proposition~\ref{prop:lebesgueumd}. Note that the inequality $\|\cdot\|_{\ell^{q_0}}\leq\|\cdot\|_{\ell^t}$ implies that \eqref{eq:pointwisesparseq} holds with $q_0$ replaced by $t$. Hence, an application of Proposition~\ref{prop:ptwisetoform} and  Theorem~\ref{thm:weightcor} with $q=t$ and $s=\infty$ proves the result.
\end{proof}

Note that in case $t\in[q_0,\infty)$  we did not need to assume the pointwise sparse domination \eqref{eq:pointwisesparseq} in our proof, but it would have sufficed to assume domination in form. For example, if we instead assumed that for an $s \in (q_0,\infty]$ and any $\vec{f},g \in L^\infty_c(\R^d)$ there exists a sparse collection $\mc{S}$ such that
\begin{equation}\label{eq:formsparseq}
\nrmb{T(\vec{f}\hspace{2pt})\cdot g}_{L^{q_0}(\R^d)}\leq C_T \has{\sum_{Q\in\mc{S}}\Big(\prod_{j=1}^m\langle f_j\rangle_{r_j,Q}\Big)^{q_0}\langle g\rangle^{q_0}_{\frac{1}{\frac{1}{q_0}-\frac{1}{s}},Q}|Q|}^{\frac{1}{q_0}},
\end{equation}
then exactly as in the proof we obtain \eqref{eq:pointwisesparsecons} for $t\in[q_0,\infty)$ with
\begin{equation}\label{eq:gamma}
\gamma=\max\left\{\frac{\frac{1}{\vec{r}}}{\frac{1}{\vec{r}}-\frac{1}{\vec{p}}},\frac{\frac{1}{q_0}-\frac{1}{s}}{\frac{1}{p}-\frac{1}{s}}\right\}.
\end{equation}
However, at this point it is not clear how to deal with the cases $t\in(r,q_0]$. In the case of \eqref{eq:pointwisesparseq} we can simply apply the estimate $\|\cdot\|_{\ell^{q_0}}\leq\|\cdot\|_{\ell^t}$ to obtain the domination required to complete the argument. However, if we  only assume the sparse domination in form \eqref{eq:formsparseq}, it is unknown whether we automatically also have \eqref{eq:formsparseq} with $q_0$ replaced by a smaller exponent $0<q\leq q_0$, meaning that it is not clear whether we have the flexibility to cover the cases $t\in(r,q_0]$ or not without \textit{assuming}  that \eqref{eq:formsparseq} also holds with $q_0$ replaced by $t$.

We point out that replacing $q_0$ by $0<q\leq q_0$ qualitatively yields the same weighted bounds, but the result is quantitatively worse in that it yields a worse exponent $\gamma$ in the bound. Thus, on all accounts it seems that the following conjecture should hold:
\begin{conjecture}[Sparse form domination implies worse sparse form domination]\label{con:sdiwsd}
Let $\vec{r}\in(0,\infty)^m$, $q_0\in(0,\infty)$, and $s\in(q_0,\infty]$. Let $T$ be an operator defined on $m$-tuples of functions and suppose that for any $f_1,\cdots,f_m \in L^\infty_c(\R^d)$ there exists a sparse collection $\mc{S}$ such that
\[
\|T(\vec{f}\hspace{2pt})\cdot g\|_{L^{q_0}(\R^d)}\lesssim \has{\sum_{Q\in\mc{S}}\Big(\prod_{j=1}^m\langle f_j\rangle_{r_j,Q}\Big)^{q_0}\langle g\rangle^{q_0}_{\frac{1}{\frac{1}{q_0}-\frac{1}{s}},Q}|Q|}^{\frac{1}{q_0}}.
\]
Then the same estimate also holds when we replace $q_0$ by any $q\in(0,q_0]$.
\end{conjecture}
We point out that even the simplest case $m=1$, $r=1$, $q_0=1$, $s=\infty$ is unknown.
For specific cases of $T$ one can usually verify the conjecture by going back to the proof of \eqref{eq:formsparseq} and insert the estimate $\nrm{\cdot}_{\ell^q} \leq \nrm{\cdot}_{\ell^{q_0}}$ at the right place in the proof. Examples where this is the case include:
\begin{itemize}
  \item In \cite[Theorem 3.5]{Lo19b} a general theorem to obtain sparse domination for an operator $T$ is shown. In this theorem  a \emph{localized $\ell^r$-estimate} is imposed on $T$ to deduce \eqref{eq:formsparseq} with $q_0=r$. The localized $\ell^r$-estimate for $T$ becomes weaker for smaller $r$, so \cite[Theorem 3.5]{Lo19b} also yields the result of Conjecture~\ref{con:sdiwsd}.
  \item One of the main results in \cite{CCDO17} is \eqref{eq:formsparseq} with $q_0 = 1$ for rough homogeneous singular operators $T_\Omega$, see also \cite{Le19} for an alternative proof. Adapting the technique in \cite{Le19}, Conjecture~\ref{con:sdiwsd} was verified for these operators in \cite[Theorem 5.1]{CLRT19}, which has implications for weighted norm inequalities for $T_\Omega$ with so-called $C_p$-weights.
\end{itemize}

To conclude this subsection, we wish to compare our result with the result \cite[Corollary 4.6]{Ni19}. Let $\vec{r}\in(0,\infty)^m$, $s\in(1,\infty]$, and let $T$ be an operator defined on $m$-tuples of functions such that for any $\vec{f} \in L^\infty_c(\R^d)$ there exists a sparse collection $\mc{S}$ such that
\[
\|T(\vec{f}\hspace{2pt})\cdot g\|_{L^1(\R^d)}\lesssim \has{\sum_{Q\in\mc{S}}\Big(\prod_{j=1}^m\langle f_j\rangle_{r_j,Q}\Big)\langle g\rangle_{s'},Q}|Q|.
\]
Then, by \cite[Corollary 4.6]{Ni19}, we find that for all exponents $\vec{p}\in(0,\infty]^m$, $\vec{t}\in(0,\infty]^m$ with $\vec{r}<\vec{p},\vec{t}$ and $p,t<s$ and all $\vec{w}\in A_{\vec{p},(\vec{r},s)}$ we have
\begin{equation*}
\nrm{\widetilde{T}}_{L_{w_1}^{p_1}(\ell^{t_1})\times\cdots\times L_{w_m}^{p_m}(\ell^{t_m})\to L_w^p(\ell^t)}\lesssim[\vec{w}]_{\vec{p},(\vec{r},s)}^{ \max\left\{\frac{\frac{1}{\vec{r}}}{\frac{1}{\vec{r}}-\frac{1}{\vec{t}}},\frac{1-\frac{1}{s}}{\frac{1}{t}-\frac{1}{s}}\right\}\cdot \max\left\{\frac{\frac{1}{\vec{r}}-\frac{1}{\vec{t}}}{\frac{1}{\vec{r}}-\frac{1}{\vec{p}}},\frac{\frac{1}{t}-\frac{1}{s}}{\frac{1}{p}-\frac{1}{s}}\right\}}.
\end{equation*}
Since
\[
\max\left\{\frac{\frac{1}{\vec{r}}}{\frac{1}{\vec{r}}-\frac{1}{\vec{p}}},\frac{1-\frac{1}{s}}{\frac{1}{p}-\frac{1}{s}}\right\}\leq\max\left\{\frac{\frac{1}{\vec{r}}}{\frac{1}{\vec{r}}-\frac{1}{\vec{t}}},\frac{1-\frac{1}{s}}{\frac{1}{t}-\frac{1}{s}}\right\}\cdot \max\left\{\frac{\frac{1}{\vec{r}}-\frac{1}{\vec{t}}}{\frac{1}{\vec{r}}-\frac{1}{\vec{p}}},\frac{\frac{1}{t}-\frac{1}{s}}{\frac{1}{p}-\frac{1}{s}}\right\},
\]
the exponent \eqref{eq:gamma} we obtain from our method improves this result in the Banach range $t\in[1,\infty)$. We point out that our method of improving this bound is exactly as was discussed in \cite[Remark 4.7]{Ni19}.
\subsection{Multilinear Calder\'on-Zygmund operators} A multilinear Calder\'on-Zygmund operators (see \cite{CR16, LN15} for the definition) satisfies the sparse domination \eqref{eq:mainthmin} for $q=1$ and $r_1=\cdots=r_m=1$, $s=\infty$, see \cite{DLP15, CR16, LN15}. Moreover, it is known that these operators have vector-valued extensions with respect to tuples $\vec{X}=(\ell^{t_1},\ldots,\ell^{t_m})$ where $\vec{t}\in(1,\infty]^m$ with $t\in(\frac{1}{m},\infty)$, see \cite{Ni19, LMMOV19}.
More precisely, multilinear Calder\'on-Zygmund operators $T$ satisfy an a.e. pointwise sparse domination result
\[
|T(\vec{f}\hspace{2pt})|\leq C_T\sum_{Q\in\mc{S}}\prod_{j=1}^m\langle f_j\rangle_{r_j,Q}\ind_Q.
\]
Hence, by applying Proposition~\ref{prop:weightsellt} with $q_0=1$ we find that for all $\vec{p}\in(0,\infty]^m$, $\vec{t}\in(1,\infty]^m$ with $p,t<\infty$ and all $\vec{w}\in A_{\vec{p},(\vec{1},\infty)}$ we have
\begin{equation*}
\nrmb{\widetilde{T}}_{L_{w_1}^{p_1}(\R^d;\ell^{t_1})\times\cdots\times L_{w_m}^{p_m}(\R^d;\ell^{t_m})\to L_w^p(\R^d;\ell^t)}\lesssim_{\vec{p},q_0,\vec{r},\vec{t}} C_T[\vec{w}]_{\vec{p},(\vec{1},\infty)}^\gamma
\end{equation*}
with
\[
\gamma=\begin{cases} \max\left\{p_1',\ldots,p_m',p\right\} &\text{if $t\in[1,\infty)$;}\\
\max\left\{p_1',\ldots,p_m',\frac{p}{t}\right\}  &\text{if $t\in(\frac{1}{m},1]$.}
\end{cases}
\]

As discussed in the previous subsection, in the case $t\in[1,\infty)$ our quantitative bound improves the previously known ones. As a matter of fact, it is sharp, since the bound is the same as the one in the scalar case.

By applying Proposition~\ref{prop:sparsequiv}, Proposition~\ref{prop:ptwisetoform}, Theorem~\ref{thm:mainthm1}, and Theorem~\ref{thm:weightcor}, the full result for the tensor extension $\widetilde{T}$ of an $m$-linear Calder\'on-Zygmund operator $T$ we obtain is as follows:
\begin{proposition}
Let $T$ be an $m$-linear Calder\'on-Zygmund operator. Let $\vec{X}$ be an $m$-tuple of quasi-Banach function spaces over $\Omega$ such that $\vec{X}^q\in\UMD_{\frac{\vec{1}}{q},\infty}$ for some  $q\in(0,1]$. Then for all simple functions $\vec{f}\in L^\infty_c(\R^d;\vec{X})$ and  $g\in L^\infty_c(\R^d)$, there exists a sparse collection of cubes $\mc{S}$ such that
\[
\nrmb{\|\widetilde{T}(\vec{f}\hspace{2pt})\|_X\cdot g}_{L^q(\R^d)}\lesssim_{\vec{X},q} C_T\has{\sum_{Q\in\mc{S}}\Big(\prod_{j=1}^m\langle \|f_j\|_{X_j}\rangle_{1,Q}\Big)^q\langle \||g|^q\|_{(X^q)^\ast}\rangle_{1,Q}|Q|}^{\frac{1}{q}}.
\]
Moreover, we have
\[
\|\widetilde{T}(\vec{f}\hspace{2pt})\|_{L_w^p(\R^d;X)}\lesssim_{\vec{X},\vec{p},q} C_T[\vec{w}]_{\vec{p},(\vec{1},\infty)}^{ \max\left\{p_1',\ldots,p_m',\frac{p}{q}\right\}}\prod_{j=1}^m\|f_j\|_{L^{p_j}_{w_j}(\R^d;X_j)}
\]
for all $\vec{p}\in(1,\infty]^m$ with $p<\infty$, all $\vec{w}\in A_{\vec{p},(\vec{1},\infty)}$, and all $\vec{f}\in L^{\vec{p}}_{\vec{w}}(\R^d;\vec{X})$.
\end{proposition}
To optimize the weighted bound, for each tuple of spaces $\vec{X}$ one should determine the largest $q\in(0,1]$ such that $\vec{X}^q\in\UMD_{\frac{\vec{1}}{q},\infty}$. For $q=1$ our bound coincides with the known sharp bound in the scalar case, so in this case our bound is optimal.

Finally, we point out that by Proposition~\ref{proposition:comparisonoldassumption} we recover the results obtained in \cite[Theorem 5.2]{LN19} where each $X_j$ was assumed to be a $\UMD$ space. In fact, we improve upon these results both in that our new bounds are quantitative as well as that we are able to handle more $m$-tuples of spaces.
In conclusion, our result recovers the full known range of vector-valued extensions of multilinear Calder\'on-Zygmund operators as well as prove new ones with new sharp weighted bounds.

\begin{remark}
  In the linear case $m=1$, the sharpness of the $T(b)$ theorem in \cite{NTV02} enabled Hyt\"onen in \cite[Theorem 3]{Hy14} to prove boundedness of the tensor extension of a Calder\'on--Zygmund operator $T$ on $L^p(\R^d;X)$ for general $\UMD$ Banach spaces $X$ from scalar-valued boundedness of $T$. It would be of great interest to develop techniques to extend more general multilinear operators beyond the function space setting.
\end{remark}
\subsection{The bilinear Hilbert transform}
The bilinear Hilbert transform is defined as
\[
\BHT(f_1,f_2)(x):=\pv\int_\R\!f_1(x-y)f_2(x+y)\,\frac{\mathrm{d}y}{y}, \qquad x \in \R.
\]
As shown in \cite{CDO18}, if $r_1,r_2,s\in(1,\infty)$ satisfy the property that there exist $\theta_1,\theta_2,\theta_3\in[0,1)$ with $\theta_1+\theta_2+\theta_3=1$ such that
\begin{equation}\label{eq:rsbht}
\frac{1}{r_1}<\frac{1+\theta_1}{2},\quad\frac{1}{r_2}<\frac{1+\theta_2}{2},\quad\frac{1}{s}>\frac{1-\theta_3}{2}
\end{equation}
or equivalently
\[
\max\big\{\frac{1}{r_1},\frac{1}{2}\big\}+\max\big\{\frac{1}{r_2},\frac{1}{2}\big\}+\max\big\{\frac{1}{s'},\frac{1}{2}\big\}<2,
\]
then for all $f_1,f_2,g\in L^\infty_c(\R)$ there is a sparse collection $\mc{S}$ such that
\[
\|\BHT(f_1,f_2)\cdot g\|_{L^1(\R)}\lesssim\sum_{Q\in\mc{S}}\langle f_1\rangle_{r_1,Q}\langle f_2\rangle_{r_2,Q}\langle g\rangle_{s',Q}|Q|.
\]
It was later shown in \cite{BM17} that for $r_1,r_2,s\in(1,\infty)$ satisfying \eqref{eq:rsbht} we also have the $\ell^q$-type sparse domination
\begin{equation}\label{eq:BHTconj}
\|\BHT(f_1,f_2)\cdot g\|_{L^q(\R)}\lesssim\Big(\sum_{Q\in\mc{S}}\langle f_1\rangle^q_{r_1,Q}\langle f_2\rangle^q_{r_2,Q}\langle g\rangle^q_{\frac{1}{\frac{1}{q}-\frac{1}{s}},Q}|Q|\Big)^{\frac{1}{q}}
\end{equation}
for all $q \in (0,s)$, verifying Conjecture~\ref{con:sdiwsd} for this operator. Hence, by Proposition~\ref{prop:sparsequiv}, Theorem~\ref{thm:mainthm1}, and Theorem~\ref{thm:weightcor}, we obtain the following result for the tensor extension $\widetilde{\BHT}$:
\begin{proposition}\label{prop:bhtresult}
Let $r_1,r_2,s\in(1,\infty)$ satisfy \eqref{eq:rsbht} and let $(X_1,X_2)$ be a pair of quasi-Banach function spaces $\Omega$ such that $\vec{X}^q \in \UMD_{\frac{\vec{r}}{q},\frac{s}{q}}$ for some $q\in(0,1]$. Then for all simple functions $\vec{f}\in L^\infty_c(\R;\vec{X})$ and  $g\in L^\infty_c(\R)$ there exists a sparse collection of cubes $\mc{S}$ such that
\begin{equation}\label{eq:bhtvvsd}
\nrmb{\|\widetilde{\BHT}(f_1,f_2)\|_X\cdot g}_{L^q(\R)}\lesssim_{\vec{X},\vec{r},s}\Big(\sum_{Q\in\mc{S}}\langle \|f_1\|_{X_1}\rangle^q_{r_1,Q}\langle \|f_2\|^q_{X_2}\rangle_{r_2,Q}\langle g\rangle^q_{\frac{1}{\frac{1}{q}-\frac{1}{s}},Q}|Q|\Big)^{\frac{1}{q}}.
\end{equation}
Moreover, we have
\[
\|\widetilde{\BHT}(f_1,f_2)\|_{L_w^p(\R;X)}\lesssim_{\vec{X},\vec{p},\vec{r},s} [\vec{w}]_{\vec{p},(\vec{r},s)}^{  \max\left\{\frac{\frac{1}{r_1}}{\frac{1}{r_1}-\frac{1}{p_1}},\frac{\frac{1}{r_2}}{\frac{1}{r_2}-\frac{1}{p_2}},\frac{\frac{1}{q}-\frac{1}{s}}{\frac{1}{p}-\frac{1}{s}}\right\}}\|f_1\|_{L^{p_1}_{w_1}(\R;X_1)}\|f_2\|_{L^{p_2}_{w_2}(\R;X_2)}
\]
for all $\vec{p}\in(0,\infty]^2$ with $\vec{r}<\vec{p}$ and $p<s$, all $\vec{w}\in A_{\vec{p},(\vec{r},s)}$, and all $\vec{f}\in L^{\vec{p}}_{\vec{w}}(\R^d;\vec{X})$.
\end{proposition}
Note that in particular we find that for all $r_1,r_2,s\in(1,\infty)$ satisfying \eqref{eq:rsbht} and all $\vec{X}\in\UMD_{\vec{r},s}$ we have
\[
\|\widetilde{\BHT}(f_1,f_2)\|_{L^p(\R;X)}\lesssim_{\vec{X},\vec{p},\vec{r},s} \|f_1\|_{L^{p_1}(\R;X_1)}\|f_2\|_{L^{p_2}(\R;X_2)}
\]
for all $f_j\in\ms{S}(\R;X_j)$.

We point out here that \cite{BM17} actually proved the vector-valued sparse domination \eqref{eq:bhtvvsd} in the cases where the $X_j$ are iterated Lebesgue spaces with the same range of exponents we obtain (see Proposition~\ref{prop:lebesgueumd}), through the helicoidal method. It is worth to note that Proposition~\ref{prop:bhtresult} extends the main result of \cite{BM17} to our more general vector spaces by only using the scalar-valued sparse domination \eqref{eq:BHTconj} as an input.

We wish to compare Proposition \ref{prop:bhtresult} with \cite[Theorem 5.1]{LN19}. First of all, in terms of weights, Proposition \ref{prop:bhtresult} improves \cite[Theorem 5.1]{LN19} by considering more general bilinear weight classes. As for the spaces, in \cite[Theorem 5.1]{LN19}, combined with  \cite[Proposition~3.3(iii) and Theorem~3.6]{LN19}, bounds in terms of pairs of quasi-Banach function spaces $(X_1,X_2)$ are obtained where there exist $\theta_1,\theta_2\in[0,1)$ with $\theta_1+\theta_2\in(0,1]$ such that
\[
\frac{1}{r_1}<\frac{1+\theta_1}{2},\quad\frac{1}{r_2}<\frac{1+\theta_2}{2},\quad\frac{1}{s_1}>\frac{\theta_1}{2},\frac{1}{s_2}>\frac{\theta_2}{2}
\]
and $((X_1^{r_1})^\ast)^{(s_1/r_1)'}$, $((X_2^{r_2})^\ast)^{(s_2/r_2)'}\in\UMD$. To compare this to our spaces, we set $\theta_3:=1-\theta_1-\theta_2\in[0,1)$ and note that
\[
\frac{1}{s}=\frac{1}{s_1}+\frac{1}{s_2}>\frac{\theta_1+\theta_2}{2}=\frac{1-\theta_3}{2}
\]
so that $(r_1,r_2,s)\in(1,\infty)$ satisfies \eqref{eq:rsbht}. Moreover, by Proposition~\ref{proposition:comparisonoldassumption} we find that for all $q\in(0,r]$ we have $\vec{X}^q\in\UMD_{\frac{\vec{r}}{q},\frac{s}{q}}$. Thus, Proposition \ref{prop:bhtresult} recovers the bounds from \cite[Theorem 5.1]{LN19}.

We point out that Proposition \ref{prop:bhtresult} implies bounds for many spaces that were not attainable in \cite{LN19}, since we are no longer restricted to requiring a $\UMD$ property on each individual space. In particular, we recover the bounds with respect to the spaces $\vec{X}=(\ell^{t_1},\ell^{t_2})$ from \cite[Corollary 4.9]{Ni19} for $t_1,t_2\in(0,\infty]$ with $\vec{t}>\vec{r}$, $t<s$.

To end this section, we compare our results to the results obtained by Amenta and Uraltsev \cite{AU19} and Di Plinio, Li, Martikainen, and Vuorinen \cite{DLMV19}. In their work they prove vector-valued bounds for $\BHT$ for triples of complex Banach spaces $(X_1,X_2,X_3)$ that are not necessarily Banach function spaces, but that are compatible in the sense that there is a bounded trilinear form $\Pi:X_1\times X_2\times X_3\to\C$. Then the trilinear form $\BHF(f_1,f_2,f_3):=\langle \BHT(f_1,f_2)f_3\rangle$ has the vector-valued analogue
\[
\BHF_\Pi(f_1,f_2,f_3):=\int_\R\!\pv\int_\R\!\Pi(f_1(x-y),f_2(x+y),f_3(x))\,\frac{\mathrm{d}y}{y}\,\mathrm{d}x,
\]
whose boundedness properties can then be studied. We point out that the main result in \cite{DLMV19} considers estimates for the same tuples of spaces as in \cite{AU19}, but for a larger range of exponents. Since our main interest is in the spaces, for simplicity we compare our result to the main result of \cite{AU19}.
To state the result we need to introduce the notion of intermediate $\UMD$ spaces. We say that a Banach space $X$ is a $u$-intermediate $\UMD$ space for $u\in[2,\infty]$ if it is isomorphic to the complex interpolation space $[E,H]_{\frac{2}{u}}$, where $E$ is a $\UMD$ space and $H$ is a Hilbert space and the couple $(E,H)$ is compatible. For $\vec{u} \in [2,\infty]^m$ We say that a tuple of Banach spaces $\vec{X}$ is $\vec{u}$-intermediate $\UMD$ if $X_j$ is $u_j$-intermediate $\UMD$ for $1\leq j\leq m$.

\begin{theorem}[{\cite[Theorem~1.1]{AU19}}]\label{thm:au19} Take $\vec{u} \in [2,\infty]^m$ and let $\vec{X}$ be a triple of $\vec{u}$-intermediate Banach spaces and let $\Pi:X_1\times X_2\times X_3\to\C$ be a bounded trilinear form. For all $p_1,p_2\in(1,\infty)$ with $p\in(1,\infty)$ satisfying
\begin{equation}\label{eq:authm1}
1<\frac{1}{u_1}\min\Big\{\frac{u_1'}{p_1'},1\Big\}+\frac{1}{u_2} \min\Big\{\frac{u_2'}{p_2'},1\Big\}+\frac{1}{u_3}\min\Big\{\frac{u_3'}{p},1\Big\},
\end{equation}
we have
\begin{equation}\label{eq:authm2}
|\BHF_\Pi(f_1,f_2,g)|\lesssim\|f_1\|_{L^{p_1}(\R;X_1)}\|f_2\|_{L^{p_2}(\R;X_2)}\|g\|_{L^{p'}(\R;X_3)}
\end{equation}
for all $f_j\in\ms{S}(\R;X_j)$, $g\in\ms{S}(\R;X_3)$.
\end{theorem}
Even though we are not able to recover any of their results for spaces that are not Banach function spaces, in the setting of Banach function spaces our results go much beyond theirs.
Indeed, consider a pair of complex quasi-Banach function spaces $(X_1,X_2)$ over $(\Omega,\mu)$. Then we define
\[
\Pi:X_1\times X_2\times X^\ast\to\C,\quad\Pi(f_1,f_2,g):=\int_\Omega\!f_1f_2g\,\mathrm{d}\mu.
\]
By an application of Fubini's Theorem, we find that for all $f_j\in\ms{S}(\R;X_j)$, $g\in\ms{S}(\R;X^\ast)$ we have
\begin{equation}\label{eq:bhftobht}
\begin{split}
|\BHF_\Pi(f_1,f_2,g)|&=\left|\int_\R\!\int_\Omega\!\BHT(f_1(\cdot,\omega),f_2(\cdot,\omega))(x)g(x,\omega)\,\mathrm{d}\mu(\omega)\,\mathrm{d}x\right|\\
&\leq\|\widetilde{\BHT}(f_1,f_2)g\|_{L^1(\R;L^1(\Omega))}.
\end{split}
\end{equation}
This means that the sparse domination result in Proposition~\ref{prop:bhtresult}  combined with Proposition~\ref{prop:vvsdequivalence} implies that whenever $r_1,r_2,s\in(1,\infty)$ satisfy \eqref{eq:rsbht} and $\vec{X}\in\UMD_{\vec{r},s}$, we obtain \eqref{eq:authm2} for all $\vec{p}\in(0,\infty]^2$ with $\vec{r}<\vec{p}$ and $p<s$, as well as weighted bounds.

Since intermediate $\UMD$ spaces are themselves $\UMD$ spaces, any of our results where $X_1$ or $X_2$ is not $\UMD$ improve on Theorem~\ref{thm:au19} in the function space setting. This includes examples such as $X_1=L^\infty(\Omega)$, $X_2=L^2(\Omega)$, or $X_1=\ell^2(\ell^\infty)$, $X_2=\ell^\infty(\ell^2)$, see Proposition~\ref{prop:lebesgueumd}.

Next, let $\vec{t}\in(0,\infty]^2$ with $\vec{r}<\vec{t}$, $1\leq t<s$ and consider the case
\[
X_1=L^{t_1}(\Omega),\quad X_2=L^{t_2}(\Omega),\quad X^\ast=L^{t'}(\Omega).
\]
Then by \eqref{eq:bhftobht} and Proposition~\ref{prop:bhtresult} with $q=1$ we obtain
\begin{equation}\label{eq:bhflebesgue}
|\BHF_\Pi(f_1,f_2,g)|\lesssim\|f_1\|_{L^{p_1}(\R;L^{t_1}(\Omega))}\|f_2\|_{L^{p_2}(\R;L^{t_2}(\Omega))}\|g\|_{L^{p'}(\R;L^{t'}(\Omega))}
\end{equation}
for all $f_j\in\ms{S}(\R;L^{t_j}(\Omega))$, $g\in\ms{S}(\R;L^{t'}(\Omega))$ and $\vec{p}\in(0,\infty]^2$ with $\vec{r}<\vec{p}$, $p<s$. This is beyond the reach of Theorem~\ref{thm:au19}, as Theorem~\ref{thm:au19} does not include Lebesgue space over non-atomic measure spaces because of the restrictions in \eqref{eq:authm1}, see  \cite[Example~6.2.3]{AU19}.

\subsection*{Acknowledgements}
The authors wish to thank Mark Veraar and Alex Amenta for their helpful feedback on the draft. Moreover, the second author is grateful to Alex Amenta for the insightful discussions regarding vector-valued extensions of the bilinear Hilbert transform. Finally, the authors express their gratitude to Francesco Di Plinio for providing some helpful historical remarks.

\newcommand{\etalchar}[1]{$^{#1}$}


\begin{thebibliography}{HNVW16}

\bibitem[ALV19]{ALV17}
A.~Amenta, E.~Lorist, and M.C. Veraar.
\newblock Rescaled extrapolation for vector-valued functions.
\newblock {\em Publ. Mat.}, 63(1):155--182, 2019.

\bibitem[AU19]{AU19}
A.~{Amenta} and G.~{Uraltsev}.
\newblock {The bilinear Hilbert transform in UMD spaces}.
\newblock arXiv:1909.06416, 2019.

\bibitem[BM17]{BM17}
C.~{Benea} and C.~{Muscalu}.
\newblock {Sparse domination via the helicoidal method}.
\newblock ArXiv:1707.05484v2, 2017.

\bibitem[BS88]{BS88}
C.~Bennett and R.~Sharpley.
\newblock {\em Interpolation of operators}, volume 129 of {\em Pure and Applied
  Mathematics}.
\newblock Academic Press, Inc., Boston, MA, 1988.

\bibitem[Bou83]{Bo83}
J.~Bourgain.
\newblock Some remarks on {B}anach spaces in which martingale difference
  sequences are unconditional.
\newblock {\em Ark. Mat.}, 21(2):163--168, 1983.

\bibitem[Bou84]{Bo84}
J.~Bourgain.
\newblock Extension of a result of {B}enedek, {C}alder\'on and {P}anzone.
\newblock {\em Ark. Mat.}, 22(1):91--95, 1984.

\bibitem[Bur83]{Bu83}
D.L. Burkholder.
\newblock A geometric condition that implies the existence of certain singular
  integrals of {B}anach-space-valued functions.
\newblock In {\em Conference on harmonic analysis in honor of {A}ntoni
  {Z}ygmund, {V}ol. {I}, {II} ({C}hicago, {I}ll., 1981)}, Wadsworth Math. Ser.,
  pages 270--286. Wadsworth, Belmont, CA, 1983.

\bibitem[Cal64]{Ca64}
A.P. Calder{\'o}n.
\newblock Intermediate spaces and interpolation, the complex method.
\newblock {\em Studia Math.}, 24:113--190, 1964.

\bibitem[CLRT19]{CLRT19}
J.~{Canto}, K.~{Li}, L.~{Roncal}, and O.~{Tapiola}.
\newblock {$C_p$ estimates for rough homogeneous singular integrals and sparse
  forms}.
\newblock arXiv:1909.08344, 2019.

\bibitem[CCDO17]{CCDO17}
J.M. {Conde-Alonso}, A.~Culiuc, F.~{Di Plinio}, and Y.~Ou.
\newblock A sparse domination principle for rough singular integrals.
\newblock {\em Anal. PDE}, 10(5):1255--1284, 2017.

\bibitem[CR16]{CR16}
J.M. {Conde-Alonso} and G.~Rey.
\newblock A pointwise estimate for positive dyadic shifts and some
  applications.
\newblock {\em Math. Ann.}, 365(3-4):1111--1135, 2016.

\bibitem[CDO18]{CDO18}
A.~{Culiuc}, F.~{Di Plinio}, and Y.~{Ou}.
\newblock Domination of multilinear singular integrals by positive sparse
  forms.
\newblock {\em J. Lond. Math. Soc. (2)}, 98(2):369--392, 2018.

\bibitem[CDO17]{CDO17}
A.~Culiuc, F.~{Di Plinio}, and Y.~Ou.
\newblock A sparse estimate for multisublinear forms involving vector-valued
  maximal functions.
\newblock In {\em Bruno {P}ini {M}athematical {A}nalysis {S}eminar 2017},
  volume~8 of {\em Bruno Pini Math. Anal. Semin.}, pages 168--184. Univ.
  Bologna, Alma Mater Stud., Bologna, 2017.

\bibitem[DLP15]{DLP15}
W.~Dami\'an, A.K. Lerner, and C.~P\'erez.
\newblock Sharp weighted bounds for multilinear maximal functions and
  {C}alder\'on-{Z}ygmund operators.
\newblock {\em J. Fourier Anal. Appl.}, 21(1):161--181, 2015.

\bibitem[DK19]{DK17b}
Luc Deleaval and Christoph Kriegler.
\newblock Dimension free bounds for the vector-valued {H}ardy-{L}ittlewood
  maximal operator.
\newblock {\em Rev. Mat. Iberoam.}, 35(1):101--123, 2019.

\bibitem[DLMV19]{DLMV19}
F.~{Di Plinio}, K.~{Li}, H.~{Martikainen}, and E.~{Vuorinen}.
\newblock {Banach-valued multilinear singular integrals with modulation
  invariance}.
\newblock arXiv:1909.07236, 2019.

\bibitem[FG91]{FG91}
D.L. Fernandez and J.B. Garcia.
\newblock Interpolation of {O}rlicz-valued function spaces and {U}.{M}.{D}.
  property.
\newblock {\em Studia Math.}, 99(1):23--40, 1991.

\bibitem[GMT93]{GMT93}
J.~{Garc{\'{\i}}a-Cuerva}, R.~Mac{\'{\i}}as, and J.L. Torrea.
\newblock The {H}ardy-{L}ittlewood property of {B}anach lattices.
\newblock {\em Israel J. Math.}, 83(1-2):177--201, 1993.

\bibitem[HL19]{HL17}
T.S. H\"{a}nninen and E.~Lorist.
\newblock Sparse domination for the lattice {H}ardy--{L}ittlewood maximal
  operator.
\newblock {\em Proc. Amer. Math. Soc.}, 147(1):271--284, 2019.

\bibitem[Hyt12]{Hy12}
T.P. Hyt\"onen.
\newblock The sharp weighted bound for general {C}alder\'on-{Z}ygmund
  operators.
\newblock {\em Ann. of Math.}, 175(3):1473--1506, 2012.

\bibitem[Hyt14]{Hy14}
T.P. Hyt\"{o}nen.
\newblock The vector-valued nonhomogeneous {T}b theorem.
\newblock {\em Int. Math. Res. Not. IMRN}, (2):451--511, 2014.

\bibitem[HK12]{HK12}
T.P. Hyt\"{o}nen and A.~Kairema.
\newblock Systems of dyadic cubes in a doubling metric space.
\newblock {\em Colloq. Math.}, 126(1):1--33, 2012.

\bibitem[HNVW16]{HNVW16}
T.P. Hyt\"onen, J.M.A.M.~van Neerven, M.C. Veraar, and L.~Weis.
\newblock {\em Analysis in {B}anach Spaces. {V}olume {I}: {M}artingales and
  {L}ittlewood-{P}aley Theory}, volume~63 of {\em Ergebnisse der Mathematik und
  ihrer Grenzgebiete.}
\newblock Springer, 2016.

\bibitem[Kal84]{Ka84}
N.J. Kalton.
\newblock Convexity conditions for non-locally convex lattices.
\newblock {\em Glasgow Math. J.}, 25(2):141--152, 1984.

\bibitem[KLW19]{KLW20}
N.J. Kalton, E.~Lorist, and L.~Weis.
\newblock Euclidean structures and operator theory in {B}anach spaces.
\newblock arXiv:1912.09347, 2019.

\bibitem[Ler13]{Le13a}
A.K. Lerner.
\newblock A simple proof of the {$A_2$} conjecture.
\newblock {\em Int. Math. Res. Not. IMRN}, (14):3159--3170, 2013.

\bibitem[Ler19]{Le19}
A.K. Lerner.
\newblock A weak type estimate for rough singular integrals.
\newblock {\em Rev. Mat. Iberoam.}, 35(5):1583--1602, 2019.

\bibitem[LN18]{LN15}
A.K. Lerner and F.~Nazarov.
\newblock Intuitive dyadic calculus: The basics.
\newblock {\em Expositiones Mathematicae}, 2018.

\bibitem[LOP{\etalchar{+}}09]{LOPTT09}
A.K. Lerner, S.~Ombrosi, C.~P\'erez, R.H. Torres, and R.~Trujillo-Gonz\'alez.
\newblock New maximal functions and multiple weights for the multilinear
  {C}alder\'on-{Z}ygmund theory.
\newblock {\em Adv. Math.}, 220(4):1222--1264, 2009.

\bibitem[LMM{\etalchar{+}}19]{LMMOV19}
K.~{Li}, J.M. {Martell}, H.~{Martikainen}, S.~{Ombrosi}, and {Vuorinen} E.
\newblock {End-point estimates, extrapolation for multilinear Muckenhoupt
  classes, and applications}.
\newblock arXiv:1902.04951, 2019.

\bibitem[LMO18]{LMO18}
K.~{Li}, J.M. {Martell}, and S.~{Ombrosi}.
\newblock {Extrapolation for multilinear Muckenhoupt classes and applications
  to the bilinear Hilbert transform}.
\newblock arXiv:1802.03338, 2018.

\bibitem[LMS14]{LMS14}
K.~Li, K.~Moen, and W.~Sun.
\newblock The sharp weighted bound for multilinear maximal functions and
  {C}alder\'on-{Z}ygmund operators.
\newblock {\em J. Fourier Anal. Appl.}, 20(4):751--765, 2014.

\bibitem[LVY19]{LVY18}
N.~Lindemulder, M.C. Veraar, and I.S. Yaroslavtsev.
\newblock The {UMD} property for {M}usielak--{O}rlicz spaces.
\newblock In {\em Positivity and noncommutative analysis -- Festschrift in
  honour of Ben de Pagter on the occasion of his 65th birthday}, Trends in
  Mathematics. Birkh\"{a}user Verlag, 2019.

\bibitem[LT79]{LT79}
J.~Lindenstrauss and L.~Tzafriri.
\newblock {\em Classical {B}anach spaces. {II}}, volume~97 of {\em Ergebnisse
  der Mathematik und ihrer Grenzgebiete}.
\newblock Springer-Verlag, Berlin-New York, 1979.

\bibitem[Lor19]{Lo19b}
E.~Lorist.
\newblock On pointwise $\ell^r$-sparse domination in a space of homogeneous
  type.
\newblock arXiv:1907.00690, 2019.

\bibitem[LN19]{LN19}
E.~Lorist and Z.~Nieraeth.
\newblock Vector-valued extensions of operators through multilinear limited
  range extrapolation.
\newblock {\em J. Fourier Anal. Appl.}, 25(5):2608--2634, 2019.

\bibitem[Loz69]{Lo69}
G.Ya. Lozanovskii.
\newblock On some {B}anach lattices.
\newblock {\em Siberian Mathematical Journal}, 10(3):419--431, 1969.

\bibitem[NTV02]{NTV02}
F.~Nazarov, S.~Treil, and A.~Volberg.
\newblock Accretive system {$Tb$}-theorems on nonhomogeneous spaces.
\newblock {\em Duke Math. J.}, 113(2):259--312, 2002.

\bibitem[NVW15]{NVW15}
J.M.A.M.~van Neerven, M.C. Veraar, and L.~Weis.
\newblock On the {$R$}-boundedness of stochastic convolution operators.
\newblock {\em Positivity}, 19(2):355--384, 2015.

\bibitem[Nie19]{Ni19}
Z.~Nieraeth.
\newblock Quantitative estimates and extrapolation for multilinear weight
  classes.
\newblock {\em Math. Ann.}, 2019.

\bibitem[O'N63]{On63}
R.~O'Neil.
\newblock Convolution operators and {$L(p,\,q)$} spaces.
\newblock {\em Duke Math. J.}, 30:129--142, 1963.

\bibitem[O'N65]{On65}
R.~O'Neil.
\newblock Fractional integration in {O}rlicz spaces. {I}.
\newblock {\em Trans. Amer. Math. Soc.}, 115:300--328, 1965.

\bibitem[Pis16]{Pi16}
G.~Pisier.
\newblock {\em Martingales in {B}anach spaces}, volume 155 of {\em Cambridge
  Studies in Advanced Mathematics}.
\newblock Cambridge University Press, Cambridge, 2016.

\bibitem[PSX12]{PSX12}
D.~Potapov, F.~Sukochev, and Q.~Xu.
\newblock On the vector-valued {L}ittlewood-{P}aley-{R}ubio de {F}rancia
  inequality.
\newblock {\em Rev. Mat. Iberoam.}, 28(3):839--856, 2012.

\bibitem[{Rub}86]{Ru86}
J.L. {Rubio de Francia}.
\newblock Martingale and integral transforms of {B}anach space valued
  functions.
\newblock In {\em Probability and {B}anach spaces ({Z}aragoza, 1985)}, volume
  1221 of {\em Lecture Notes in Math.}, pages 195--222. Springer, Berlin, 1986.

\bibitem[Sch10]{Sc10}
A.~Schep.
\newblock Products and factors of {B}anach function spaces.
\newblock {\em Positivity}, 14(2):301--319, 2010.

\bibitem[Tri78]{Tr78}
H.~Triebel.
\newblock {\em Interpolation theory, function spaces, differential operators},
  volume~18 of {\em North-Holland Mathematical Library}.
\newblock North-Holland Publishing Co., Amsterdam-New York, 1978.

\end{thebibliography}
\end{document}